\documentclass[11pt,a4paper,twoside]{article}
\usepackage{latexsym,amsfonts,amsmath,amsthm,amssymb}

\usepackage[utf8x]{inputenc}
\usepackage[T1]{fontenc} 

\usepackage{fullpage}

\usepackage{xcolor}
   \definecolor{MyFillColor}{rgb}{0.9,0.9,0.9} 
   \definecolor{MyLinesColor}{rgb}{0.9,0.9,0.9} 

\usepackage{tikz}
\usetikzlibrary{patterns}
\usepackage{tikz-cd}

\tikzset{
	hatch distance/.store in=\hatchdistance,
	hatch distance=10pt,
	hatch thickness/.store in=\hatchthickness,
	hatch thickness=2pt
}

\makeatletter
\pgfdeclarepatternformonly[\hatchdistance,\hatchthickness]{MyPattern} 
{\pgfqpoint{0pt}{0pt}}
{\pgfqpoint{\hatchdistance}{\hatchdistance}}
{\pgfpoint{\hatchdistance-1pt}{\hatchdistance-1pt}}%
{
	\pgfsetcolor{\tikz@pattern@color}
	\pgfsetlinewidth{\hatchthickness}
	\pgfpathmoveto{\pgfqpoint{0pt}{0pt}}
	\pgfpathlineto{\pgfqpoint{\hatchdistance}{\hatchdistance}}
	\pgfusepath{stroke}
}

\pgfdeclarepatternformonly[\hatchdistance,\hatchthickness]{MyPatternOther} 
{\pgfqpoint{0pt}{0pt}}
{\pgfqpoint{\hatchdistance}{\hatchdistance}}
{\pgfpoint{\hatchdistance-1pt}{\hatchdistance-1pt}}%
{
	\pgfsetcolor{\tikz@pattern@color}
	\pgfsetlinewidth{\hatchthickness}
	\pgfpathmoveto{\pgfqpoint{\hatchdistance}{0pt}}
	\pgfpathlineto{\pgfqpoint{0pt}{\hatchdistance}}
	\pgfusepath{stroke}
}

\makeatother

\usepackage{caption} 
\usepackage{subcaption} 

\usepackage{enumitem}
   \setlist{nosep}

\usepackage{mathrsfs}
\usepackage{fourier} 

\usepackage[final]{microtype} 
\usepackage[hidelinks,pagebackref]{hyperref} 
\hypersetup{final} 

\usepackage{mathtools}

\newcommand*{\R}{\mathbb R} 
\newcommand*{\N}{\mathbb N} 
\newcommand*{\Sc}{\mathcal{S}} 

\newcommand*{\cHPV}{c} 
\newcommand*{\critmax}{\alpha} 
\newcommand*{\critmin}{\beta} 

\DeclareMathOperator{\conv}{conv}
\DeclareMathOperator{\id}{id}
\DeclareMathOperator{\codim}{codim}
\DeclareMathOperator{\rank}{rank}

\newtheorem{theorem}{Theorem}[section]
\newtheorem{corollary}[theorem]{Corollary}
\newtheorem{lemma}[theorem]{Lemma}

\newtheorem{proposition}[theorem]{Proposition}
\newtheorem{conjecture}[theorem]{Conjecture}

\newtheorem{thmalpha}{Theorem}

{
\theoremstyle{definition}

\newtheorem{remark}[theorem]{Remark}

}

\allowdisplaybreaks

\begin{document}

\title{\bf Approximation, Gelfand, and Kolmogorov numbers of Schatten class embeddings}

\author{Joscha Prochno \and Micha{\l} Strzelecki}

\date{March 23, 2021}

\maketitle

\begin{abstract}
Let $0<p,q\leq \infty$ and denote by $\Sc_p^N$ and $\Sc_q^N$ the corresponding Schatten classes of real $N\times N$ matrices. We study approximation quantities of natural identities $\Sc_p^N\hookrightarrow \Sc_q^N$ between Schatten classes and prove asymptotically sharp bounds up to constants only depending on $p$ and $q$, showing how approximation numbers are intimately related to the Gelfand numbers and their duals, the Kolmogorov numbers. In particular, we obtain new bounds for those sequences of $s$-numbers. Our results improve and complement bounds previously obtained by B. Carl and A. Defant [J. Approx. Theory, 88(2):228–256, 1997], Y. Gordon, H. K\"onig, and C. Sch\"utt [J. Approx. Theory, 49(3):219–239, 1987], A. Hinrichs and C. Michels [Rend. Circ. Mat. Palermo (2) Suppl., (76):395–411, 2005], and A. Hinrichs, J. Prochno, and J. Vyb\'iral [preprint, 2020]. We also treat the case of quasi-Schatten norms, which is relevant in applications such as low-rank matrix recovery.
\medspace
\end{abstract}

\tableofcontents

\section{Introduction and main results}

The family of all compact operators between Hilbert spaces (i.e., those operators for which the image under the operator of any bounded subset of the domain is a relatively compact subset of the codomain) with their sequence of singular values belonging to the space $\ell_p$ is known as the Schatten class $\Sc_p$ ($0<p\leq \infty$). Important subclasses are the trace class operators for $p=1$ and the Hilbert--Schmidt operators for $p=2$. The collection of Schatten classes was introduced by R. Schatten in \cite[Chapter 6]{Schatten1960}, who worked in the more general setting of symmetric gauge functions and the unitarily invariant crossnorms on the subalgebra of finite rank operators generated by them. The very origin of his work can be traced back to his paper \cite{S1946} and the subsequent works \cite{SvN1946_II,SvN1946_III} with von Neumann, studying nuclear operators on Hilbert spaces; on Banach spaces this had later been considered by Ruston \cite{R1951} and on locally convex spaces by Grothendieck \cite{G233}. Before Schatten's monograph appeared, spaces of compact operators (back then referred to as completely continuous operators) had
received comparably little attention in the literature, but today they form a classical and still very active part of modern functional analysis. In fact, Schatten classes provide the mathematical framework and foundation for topics such as low-rank matrix recovery and completion (see, e.g., \cite{CR2009, CDK2015, FR2013, KS2018, RT2011} and references cited therein) and are fundamental in quantum information theory, for instance in connection to counterexamples to Hasting's additivity conjecture (see, e.g., \cite{AS2017, ASW2010, ASW2011}). This explains the increased interest in their structure in recent years. The Schatten class $\Sc_p$ is commonly referred to as a non-commutative $\ell_p$ space, because both the space of operators and the sequence space share various structural characteristics, for instance, they are lexicographically ordered, uniformly convex for $1<p<\infty$, and satisfy a trace duality relation together with a corresponding H\"older inequality. Also, while the dual space of $c_0$ is the space $\ell_1$, the dual space of the space of compact operators on a Hilbert space is the Schatten $1$-class. However, despite several similarities on different levels, there are many differences in their analytic, geometric, and probabilistic behavior, and often arguments in the non-commutative setting are more subtle and delicate.

Today, there is a vast literature in geometric functional analysis examining both the local and global structure of Schatten classes. Among the now famous and classical works is a paper of Gordon and Lewis, showing that for $p\neq 2$ the class $\Sc_p$ does not have local unconditional structure and consequently does not possess an unconditional basis \cite{GL1974}. Before, Kwapie\'n and Pe{\l}czy\'nski had shown that $\Sc_1$ as well as $\Sc_\infty$ are not isomorphic to subspaces with an unconditional basis \cite{KP1970}. Another breakthrough result is due to Tomczak-Jaegermann, who succeeded in \cite{TJ1974} to prove that $\Sc_1$ has Rademacher cotype $2$. The past $20$ years or so have seen more work on the finite-dimensional front. For instance, Carl and Defant \cite{CD1997} proved a Garnaev--Gluskin result for Gelfand numbers of Schatten class embeddings $\Sc_p^N\hookrightarrow \Sc_2^N$ for $1\leq p\leq 2$, K\"onig, Meyer, and Pajor \cite{KMP1998} obtained that the isotropic constants of $\Sc_p^N$ unit balls are bounded above by absolute constants for all $1\leq p\leq \infty$, and Gu\'edon and Paouris studied their concentration of mass properties in \cite{GP2007}. Even more recently, Radke and Vritsiou succeeded in confirming the thin-shell conjecture for $\Sc_\infty^N$ \cite{RV2016}, Hinrichs, Prochno, and Vyb\'iral computed the entropy numbers for identity mappings $\Sc_p^N\hookrightarrow\Sc_q^N$ for all $0<p,q\leq \infty$ \cite{HPV2017} and recently obtained asymptotically sharp estimates for Gelfand numbers in almost all regimes \cite{HPV2020}, Vritsiou confirmed the variance conjecture for $\Sc_\infty^N$ \cite{V2018}, and, in a series of papers, Kabluchko, Prochno, and Th\"ale computed the exact asymptotic volume and volume ratio of $\Sc_p^N$ unit balls for $0<p\leq \infty$ \cite{KPT2020_volume_ratio}, studied the threshold behavior of the volume of intersections of unit balls \cite{KPT2020_intersections}, and obtained large deviation principles for the empirical spectral measures of random matrices in Schatten unit balls \cite{KPT2019_sanov}.

\subsection{Approximation, Gelfand, and Kolmogorov numbers for Schatten class embeddings} \label{sec:agk numbers for schatten}

A possible way to quantify the degree of compactness of an operator is via its sequence of $s$-numbers, which includes the approximation and entropy numbers as well as Gelfand and Kolmogorov numbers. The focus of this paper will be on approximation numbers for Schatten class embeddings, which, as we will show, are intimately related to their sequence of Gelfand and Kolmogorov numbers. In fact, we shall show that, depending on the relation between $p$ and $q$, the approximation numbers of the natural identities $\Sc_p^N\hookrightarrow\Sc_q^N$ either behave like the Gelfand numbers or like their duals, the Kolmogorov numbers. 

Let us continue with the definition of those $s$-numbers and of the Schatten classes before we present currently known results followed by our main findings; for unexplained notions or notation and properties as well as relations between those $s$-numbers, we refer to Section \ref{sec:prelim} below. For $0<p\le \infty$, we denote by $\Sc_p^N$ the Schatten $p$-class of real $N\times N$ matrices acting from $\ell_2^N$ to $\ell_2^N$ equipped with the Schatten $p$-(quasi-)norm
\[
\|A\|_{\Sc_p}:=\bigg(\sum_{j=1}^N\sigma_j(A)^p\bigg)^{1/p},
\]
where $\big(\sigma_j(A)\big)_{j=1}^N$ is the sequence of singular values of $A$. Given quasi-Banach spaces $X,Y$ and an operator $T\in\mathscr L(X,Y)$, we shall denote by
\[
a_n(T) := \inf\big\{ \|T-T_n\|\,:\, T_n\in\mathscr L(X,Y),\, \rank(T_n) < n \big\},\qquad n\in\N,
\]
the sequence of approximation numbers of $T$. For $n\in\N$, we define the $n$-th Gelfand number of the operator $T$ by
\[
c_n(T):= \inf\big\{ \|T|_F\| \,:\, F\subset X,\,\codim F < n \big\},
\]
where $T|_F$ denotes the restriction of the operator $T$ to the subspace $F$. Last but not least, the $n$-th Kolmogorov number of $T$ is defined as
\[
d_n(T):=\inf\big\{\|Q_E^YT\|\,:\,E\subset Y,\,\dim(E)<n\big\},
\]
where $Q_E^Y:Y\to Y/E$ denotes the quotient mapping from $Y$ onto the quotient space $Y/E$. 
Let us remark that the definition of the $n$-th Kolmogorov number can be reformulated as
\[
d_n(T)=\inf_{\substack{E\subset Y \\ \dim(E)<n}}\sup_{x\in B_X}\inf_{y\in E}\|Tx-y\|_Y.
\]

\subsubsection{Known results for approximation numbers}

In \cite{GKS1987}, Gordon, K\"onig, and Sch\"utt investigated the famous problem of self-duality of entropy numbers and proved several probabilistic results with a view towards studying $s$-numbers of linear operators between Banach spaces. In \cite[Proposition 3.7]{GKS1987} they applied their estimates to natural embeddings between Schatten classes and proved that for $1<p<2$,
\[
a_n\big(\Sc_p^N\hookrightarrow\Sc_{p^*}^N\big)
\asymp_p
\begin{cases}
1 &:\, 1\leq n \leq \lfloor N^{3-2/p}\rfloor,\\
\frac{N^{3/2-1/p}}{\sqrt{n}} &:\, \lfloor N^{3-2/p}\rfloor<n\leq \frac{N^2}{2}.
\end{cases}
\]
where $\asymp_p$ denotes equivalence up to constants only depending on $p$, $\lfloor\cdot\rfloor$ is the floor function, and $p^*$ denotes the H\"older conjugate. In the case where $n>\frac{N^2}{2}$ they obtained partial results, showing that if $\frac{N^2}{2}< n \leq N^2-c_pN^{3-2/p}\log N$ for some $p$-dependent constant $c_p\in(0,\infty)$, then
\[
C_p^{-1} N^{-1/2-1/p} \sqrt{N^2-n} \leq a_n\big(\Sc_p^N\hookrightarrow\Sc_{p^*}^N\big) \leq C_p N^{-1/2-1/p} \sqrt{N^2-n} \sqrt{\log N},
\]
for some constant $C_p\in(0,\infty)$ only depending on $p$. For the range $N^2-c_pN^{3-2/p}\log N< n < N^2$, they showed that 
\[
\max\big\{N^{1-2/p},\,C_p^{-1} N^{-1/2-1/p}\sqrt{N^2-n} \big\} \leq a_n\big(\Sc_p^N\hookrightarrow\Sc_{p^*}^N\big) \leq C_p N^{1-2/p} \log N.
\]
In \cite[Remark on p. 237]{GKS1987} the authors mention that for $n>N^2/2$ those logarithmic terms are probably not needed and, as we will show in this paper, this is indeed the case for large $n$ (see Theorem \ref{thm: main-approximation} below). In the boundary case $p=1$ (and so $p^*=\infty$), Gordon, K\"onig, and Sch\"utt determined the precise asymptotic behavior up to absolute constants, proving that
\[
a_n\big(\Sc_1^N\hookrightarrow\Sc_{\infty}^N\big) \asymp 
\begin{cases}
1 &:\, 1\leq n \leq N, \\
\sqrt{N/n} &:\, N\leq n\leq N^2/2, \\
N^{-3/2} \sqrt{N^2-n} &:\, N^2/2 \leq n \leq N^2-N,\\
N^{-1} &:\, N^2-N\leq n \leq N^2.
\end{cases}
\]
A decade later, this result was complemented by Carl and Defant in \cite{CD1997}, who showed that
\[
a_n\big(\Sc_1^N\hookrightarrow\Sc_{2}^N\big) = 
a_n\big(\Sc_2^N\hookrightarrow\Sc_{\infty}^N\big) \asymp 
\begin{cases}
1 &:\, 1\leq n \leq \lfloor N^2/2 \rfloor, \\
\sqrt{\frac{N^2-n+1}{N^2}} &:\, \lfloor N^2/2 \rfloor \leq n\leq N^2-N+1, \\
N^{-1/2} &:\,N^2-N+1\leq n\leq N^2. 
\end{cases}
\]
More generally, they obtained in \cite[Remark 2]{CD1997} that for all $1\leq n \leq N^2$ and any $2\leq q\leq \infty$,
\[
a_n\big(\Sc_2^N\hookrightarrow\Sc_{q}^N\big) \asymp_q
  \max\Bigg\{ \frac{1}{\|\Sc_q^N\hookrightarrow\Sc_{2}^N\|},\,\sqrt{\frac{N^2-n+1}{N^2}}\bigl\|\Sc_2^N\hookrightarrow\Sc_{q}^N\bigr\| \Bigg\} = \max\Bigg\{N^{1/q-1/2},\, \sqrt{\frac{N^2-n+1}{N^2}} \Bigg\}.
\] 
The other regime, with $q<p$ and where $\Sc_p^N\hookrightarrow\Sc_{q}^N$ is considered, had not been studied before Hinrichs and Michels in \cite{HM2005}.
In the case that $1\leq q\leq p\leq \infty$ and $1\leq n \leq N^2$, the authors proved \cite[Proposition 4.1, Corollary 4.8]{HM2005} that
\[
a_n\big(\Sc_p^N\hookrightarrow\Sc_{q}^N\big) \asymp \max\Bigl\{1, \frac{N^2-n+1}{N}\Bigr\}^{1/q-1/p},
\]
which for $p=2$ was independently obtained in \cite{DMM2006} via abstract interpolation methods and summing norm estimates. 

\subsubsection{Known results for Gelfand numbers}
\label{subsec:known results Gelfand}

The most recent work that complements the body of research on $s$-numbers for Schatten class embeddings, and in particular previous works of Carl and Defant \cite{CD1997}, Hinrichs and Michels \cite{HM2005}, and Ch\'avez-Dom\'inguez and Kutzarova \cite{CDK2015}, is a paper by Hinrichs, Prochno, and Vyb\'iral who computed in almost all remaining cases the Gelfand numbers for natural embeddings between Schatten classes \cite{HPV2020}. In \cite[Theorem A]{HPV2020} they present the following asymptotics:
if $0<p,q\,\le\,\infty$ and $n,N\in\N$ with $1\leq n \leq N^2$, then
\[
c_n\big(\Sc_p^N\hookrightarrow \Sc_q^N\big)  
\asymp_{p,q}
\begin{cases}
\max\left\{1,\frac{N^2 - n +1}{N} \right\} ^{1/q-1/p} & :\, \ 0<q \le p \le \infty, \\
\min\bigg\{ 1,\frac{N}{n} \bigg\}^{1/p-1/q}
& :\, \ 0<p\leq 1 \text{ and } p<q\leq 2,\\
\min\Biggl\{1,\frac{N^{3/2-1/p}}{n^{1/2}}\Biggr\}^{\frac{1/p-1/q}{1/p-1/2}}
& :\, \ 1\leq p \leq q\leq 2,\\
\min\Big\{1,\frac{N^{3/2-1/p}}{n^{1/2}}\Big\} & :\, \ 1 < p \le 2 \le q\le \infty \text{ and } 1 \leq n \leq c_{p,q} N^2,\\
1 & :\, \ 2 \le p \le q \le \infty \text{ and } 1 \leq n \leq c_{p,q} N^2,\\
N^{1/q-1/p} & :\, \ 0 < p \le q \le \infty \text{ and }  N^2 - \cHPV N^{1+2/q}+1 \leq n \leq  N^2. 
\end{cases}
\]
Here $c_{p,q} \in (0,1)$ is a constant depending on $p$ and $q$ and $\cHPV \in(0,1)$ is an absolute constant. As the authors explain, the above asymptotics cover almost all cases.
	In the Banach space setting,
only for the intermediate range $c_{p,q} N^2\le n \le N^2 - \cHPV N^{1+2/q}+1$ 
in the cases  $1\leq p < 2 < q \leq \infty$ and $2 < p  \le q \le \infty$  
there remains some gap.
In the quasi-Banach case $0<p\leq 1$ and $q\geq 2$ there remains some gap in upper and lower bounds in the ranges of small and intermediate codimensions. In those cases (see \cite[Theorem B]{HPV2020}) the authors presented the following estimates: 
\begin{enumerate}
\item If\, $0<p\leq 1$, $2\leq q\leq \infty$, and $1\leq n \leq c_{p,q}N^2$, then
\[
\min\bigg\{ 1,\frac{N}{n} \bigg\}^{1/p-1/q} \lesssim_{p,q} c_n\big(\Sc_p^N\hookrightarrow \Sc_q^N\big)  
\lesssim_{p}\min\Big\{1,\frac{N}{n} \Big\}^{1/p-1/2},
\]
which is sharp up to constants for $q=2$.
\item If\, $0<p \le 1,2  \le q \le \infty$ and $ c_{p,q} N^2 \le n \le N^2 - \cHPV N^{1+2/q}+1$, then
\[
\min\bigg\{ 1,\frac{N}{n} \bigg\}^{1/p-1/q} \lesssim_{p,q} c_n\big(\Sc_p^N\hookrightarrow \Sc_q^N\big)  
\lesssim_{q} N^{-1/p-1/2} (N^2 - n +1)^{1/2},
\] 
where the upper bound remains valid as long as $0<p \le 2  \le q \le \infty$.

\item If\, $2 \le p  \le q \le \infty$ and $N^2 - c_q^{-2} \, N^{1+2/p} +1  \le n \le N^2 - \cHPV N^{1+2/q}+1$, then 
\[
\sqrt{\frac{N^2-n+1}{N^2}}^{\,\frac{1/p-1/q}{1/2-1/q}}
 \lesssim
 c_n\big(\Sc_p^N\hookrightarrow \Sc_q^N\big) 
 \lesssim_{q} N^{1/2-1/p} \sqrt{\frac{N^2-n+1}{N^2}},
\] 
which is sharp up to constants for $p=2$. Note that when $1\leq n \leq N^2 - c_q^{-2} \, N^{1+2/p} +1$, then the previous upper bound is replaced by the trivial upper bound $1$. 
\end{enumerate}
Here $c_q,c_{p,q} \in (0,1)$ are constants depending on $p$ and/or $q$ and $c\in(0,1)$ is an absolute constant.

\subsection{Main results}
\label{sec:agk numbers for schatten main results}


Below $\cHPV \in(0,1)$ denotes the universal constant from \cite[Lemma 2.5]{HPV2020}.
We use $c(p)$, $c(p,q)$, etc.\ for positive constants which depend only on the parameters given and the value of which may change from line to line.


The first result concerns the Gelfand numbers of natural identities between Schatten classes and closes a gap in \cite{HPV2020}.
More precisely, we shall show in Proposition \ref{prop:Gelfand-missing-intermediate-regime} below that the upper bounds obtained in \cite[Propositions 4.5 and 4.8]{HPV2020} (see also 2. in Subsection \ref{subsec:known results Gelfand} above) for intermediate sized codimensions and $0 < p \leq 2 \leq q \leq \infty$ are  sharp by providing an asymptotically matching lower bound.

\begin{thmalpha}[Main result for Gelfand numbers]
	\label{thm:main-Gelfand}
	Let $0 < p \leq 2 \leq q \leq \infty$ and assume that $n,N\in \N$ with $1\leq n \leq N^2$.
	Then, whenever $  (1-\cHPV) N^2 \leq n\leq N^2 - \cHPV N^{1+2/q} +1$,
	\begin{equation*}
		N^{-1/2-1/p} \sqrt{N^2-n+1} \lesssim c_n\big(\Sc_p^N \hookrightarrow \Sc_q^N\big) 
		\lesssim_q 	N^{-1/2-1/p} \sqrt{N^2-n+1}
	\end{equation*}
	with $\cHPV\in(0,1)$ being the universal constant from \cite[Lemma 2.5]{HPV2020}. If $1\leq p \leq 2 \leq q \leq \infty$, then we can obtain the lower bound with explicit constant $1$ instead of some absolute constant.
\end{thmalpha}

\begin{remark}
For $2=p \leq q \leq \infty$, we shall also show that the lower bound of Hinrichs and Michels (see 3. above) can be obtained with explicit constant $1$ instead of an unknown absolute constant.
\end{remark}


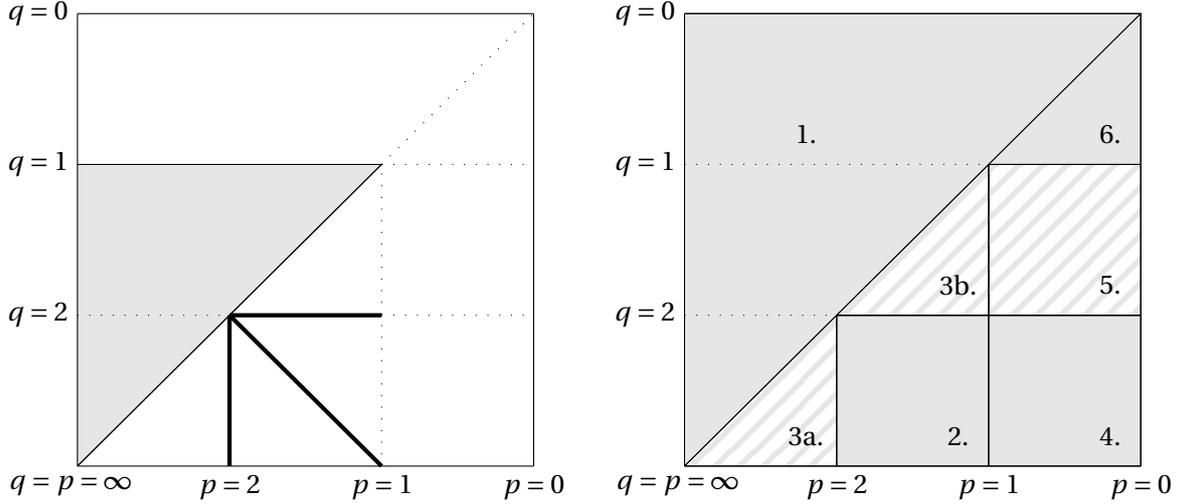
\begin{figure}[h]
\centering
		\begin{tikzpicture}[scale=2]
		\draw [fill=MyFillColor] (0,0) -- (2,2) -- (0,2);
		\draw (0,0) -- (0,3);
		\draw (0,0) -- (3,0);
		\draw (0,3) -- (3,3);
		\draw (3,0) -- (3,3);
		\draw (0,0) -- (2,2);
		\draw [loosely dotted] (2,2) -- (3,3);
		\draw [ultra thick] (1,0) -- (1,1);
		\draw [loosely dotted] (2,0) -- (2,2);
		\draw [loosely dotted] (0,1) -- (1,1);
		\draw [ultra thick] (1,1) -- (2,1);
		\draw [loosely dotted] (2,1) -- (3,1);
		\draw [loosely dotted] (2,2) -- (3,2);
		\draw [ultra thick] (1,1) -- (2,0);
		\node [below] at (-0.05,0) {$q=p=\infty$};
		\node [below] at (1,0) {$p=2$};
		\node [below] at (2,0) {$p=1$};
		\node [below] at (3,0) {$p=0$};		
		\node [left] at (0,1) {$q=2$};
		\node [left] at (0,2) {$q=1$};
		\node [left] at (0,3)  {$q=0$};
		\end{tikzpicture}
		\quad
		\begin{tikzpicture}[scale=2]
		\draw [fill=MyFillColor] (0,0) -- (3,3) -- (0,3);
		\draw [fill=MyFillColor] (1,0) rectangle (2,1);
		\draw [fill=MyFillColor] (2,0) rectangle (3,1);
		\draw [fill=MyFillColor] (2,2) -- (3,2) -- (3,3);
		\draw [pattern=MyPattern, pattern color=MyLinesColor] (0,0) -- (1,0) -- (1,1);
		\draw [pattern=MyPattern, pattern color=MyLinesColor] (1,1) -- (2,1) -- (2,2);
		\draw [pattern=MyPattern, pattern color=MyLinesColor] (2,1) rectangle (3,2);
		\draw (0,0) -- (0,3);
		\draw (0,0) -- (3,0);
		\draw (0,3) -- (3,3);
		\draw (3,0) -- (3,3);
		\draw (0,0) -- (3,3);
		\draw (1,0) -- (1,1);
		\draw (2,0) -- (2,2);
		\draw [loosely dotted] (0,1) -- (1,1);
		\draw (1,1) -- (3,1);
		\draw (2,2) -- (3,2);
		\draw [loosely dotted] (0,2) -- (2,2);
		\node [below] at (-0.05,0) {$q=p=\infty$};
		\node [below] at (1,0) {$p=2$};
		\node [below] at (2,0) {$p=1$};
		\node [below] at (3,0) {$p=0$};		
		\node [left] at (0,1) {$q=2$};
		\node [left] at (0,2) {$q=1$};
		\node [left] at (0,3)  {$q=0$};
		\node [] at (0.8,2.2) {\ref{item-appr-upper-triangle}.};
		\node [] at (1.8,0.2) {\ref{item-appr-square}.};
		\node [] at (0.8,0.2) {\ref{item-appr-two-triangles-a}.};
		\node [] at (1.8,1.2) {\ref{item-appr-two-triangles-b}.};
		\node [] at (2.8,0.2) {\ref{item-appr-quasi-low}.};	
		\node [] at (2.8,1.2) {\ref{item-appr-quasi-up}.};	
		\node [] at (2.8,2.2) {\ref{item-appr-quasi-1}.};		
		\end{tikzpicture}
	\captionsetup{singlelinecheck=off}
	\caption[foo]{Summary of the known results for approximation numbers before (left) and now  from Theorem~\ref{thm: main-approximation} (right).
		Gray regions and thick lines correspond to exact asymptotics (up to logarithms in the case of~\ref{item-appr-square} and the diagonal line $p=q^*$), lined regions to exact asymptotics for small and large $n$, but nonmatching upper and lower bounds in the intermediate range of $n$.
		More precisely:
		\begin{itemize}
			\item[] \ref{item-appr-upper-triangle} --  $a_n(\Sc_p^N\hookrightarrow \Sc_q^N) \asymp_{p,q} \max\bigl\{ 1, \frac{N^2-n+1}{N}\bigr\}^{1/q-1/p}$;
			\item[] \ref{item-appr-square} -- logarithms removed for large $n$, but remain in part of the intermediate range;	
			\item[]	\ref{item-appr-two-triangles-a}, \ref{item-appr-two-triangles-b} -- nonmatching upper and lower bounds in the intermediate range of $n$;
			\item[]	\ref{item-appr-quasi-low} -- no dependence on $p$: $a_n(\Sc_p^N\hookrightarrow \Sc_q^N) = a_n(\Sc_1^N\hookrightarrow \Sc_q^N)$, exact asymptotics known;
			\item[] \ref{item-appr-quasi-up} --  no dependence on $p$: $a_n(\Sc_p^N\hookrightarrow \Sc_q^N) = a_n(\Sc_1^N\hookrightarrow \Sc_q^N)$, nonmatching upper and lower bounds in the intermediate range of $n$ as in \ref{item-appr-two-triangles-b};
			\item[] \ref{item-appr-quasi-1} -- $a_n(\Sc_p^N\hookrightarrow \Sc_q^N) =1$.
		\end{itemize}
	}   
	\label{fig:approximation results} 
\end{figure}


We continue with our main result on the approximation numbers of natural embeddings of Schatten classes.
For $1  \leq p \leq 2 \leq q \leq \infty$ we define
\begin{equation*}
\critmax_{p,q} := \max\bigl\{ 3-2/p,1+2/q\bigr\}\qquad\text{and}\qquad
\critmin_{p,q}:= \min\bigl\{ 3-2/p,1+2/q\bigr\}.
\end{equation*}
Note that $1\leq \critmin_{p,q} \leq \critmax_{p,q} \leq 2$ and $\critmax_{p,q}=\critmin_{p,q} = 3-2/p$ whenever $1/p+1/q=1$.

In the statement below we also include  all previously known results.
The state of the art and our contribution are summarized in Figure~\ref{fig:approximation results}.


\begin{thmalpha}[Main result for approximation numbers]\label{thm: main-approximation}
Let $0<p,q\leq \infty$ and assume that $n,N\in\N$ with $1\leq n \leq N^2$.
Let $\cHPV \in(0,1)$ be the universal constant from \cite[Lemma 2.5]{HPV2020}.
Then the following estimates hold (see Figure~\ref{fig:approximation results}).
\begin{enumerate}
\item
\label{item-appr-upper-triangle}
 If $0<q\leq p \leq \infty$, then
\[
a_n\big(\Sc_p^N\hookrightarrow \Sc_q^N\big)
\asymp_{p,q} 
\max\Bigl\{ 1, \frac{N^2-n+1}{N}\Bigr\}^{1/q-1/p}.
\]
\item
\label{item-appr-square}
If $1\leq p\leq 2\leq q\leq \infty$, then
\begin{align*}
a_n\big(\Sc_p^N\hookrightarrow \Sc_q^N\big)
\asymp_{p,q} 
\begin{cases}
\min\Bigl\{ 1, \frac{N^{\critmax_{p,q}/2}}{n^{1/2}}\Bigr\} 
     &:\, 1 \leq n\leq (1-\cHPV)N^2, \\
N^{\critmax_{p,q}/2-2} \sqrt{N^2-n+1}
     &:\, N^2 - \cHPV N^{\critmax_{p,q}} +1 \leq n\leq N^2 -  \cHPV N^{\critmin_{p,q}}+1,\\
N^{1/q-1/p}
     &:\,  N^2 -\cHPV N^{\critmin_{p,q}} +1\leq n \leq N^2.
\end{cases}
\end{align*}
Moreover, for $(1-\cHPV) N^2 \leq n\leq N^2-\cHPV N^{\critmax_{p,q}} + 1$, we have the  bounds
\[
N^{\critmax_{p,q}/2-2} \sqrt{N^2-n+1}
 \lesssim_{p,q}
  a_n\big(\Sc_p^N\hookrightarrow \Sc_q^N\big)
\lesssim_{p,q}
N^{\critmax_{p,q}/2-2} \sqrt{N^2-n+1} \sqrt{\log N}, 
\]
where the logarithm on the right-hand side can be removed if $p=1$ or $q=\infty$.
\item
\label{item-appr-two-triangles}
\begin{enumerate}
	\item 
	\label{item-appr-two-triangles-a}
	If $2\leq p\leq q\leq \infty$, then 
\[
a_n\big(\Sc_p^N \hookrightarrow \Sc_q^N\big)
\lesssim_{q}
\begin{cases}
1       
           &:\,  1 \leq n \leq N^2 - c(q) N^{1+2/p} +1,\\
N^{-1/2-1/p} \sqrt{N^2-n+1}
          &:\, N^2 - c(q) N^{1+2/p} +1\leq n\leq N^2 -  \cHPV N^{1+2/q}+1,\\
N^{1/q-1/p} 
          &:\, N^2 - \cHPV N^{1+2/q} +1\leq n \leq N^2.
\end{cases}
\]
and
\[
a_n\big(\Sc_p^N \hookrightarrow \Sc_q^N\big)
\gtrsim_{p,q}
\begin{cases}
1       
           &:\,  1 \leq n \leq (1-\cHPV)N^2,\\
\sqrt{\frac{N^2-n+1}{N^2}}^{\frac{1/p-1/q}{1/2-1/q}}
          &:\, (1-\cHPV)N^2 \leq n\leq N^2 - \cHPV N^{1+2/q}+1,\\
N^{1/q-1/p} 
          &:\, N^2 - \cHPV N^{1+2/q} +1\leq n \leq N^2.
\end{cases}
\]
Here $c(q)\in(0,\infty)$ is a constant that depends only on $q$.
	\item
	\label{item-appr-two-triangles-b}
	 If $1\leq p\leq q\leq 2$, then  $a_n\big(\Sc_p^N\hookrightarrow \Sc_q^N\big)  = a_n\big(\Sc_{q^*}^N\hookrightarrow \Sc_{p^*}^N\big)$ and so, by item~\ref{item-appr-two-triangles-a}, we get the estimates 
	\[
a_n\big(\Sc_p^N \hookrightarrow \Sc_q^N\big)
\lesssim_{p}
\begin{cases}
1       
           &:\,  1 \leq n \leq N^2 - c(p) N^{3-2/q} +1,\\
N^{-3/2 + 1/q} \sqrt{N^2-n+1}
          &:\, N^2 - c(p) N^{3-2/q} +1\leq n\leq N^2 -  \cHPV N^{3-2/p}+1,\\
N^{1/q-1/p} 
          &:\, N^2 - \cHPV N^{3-2/p} +1\leq n \leq N^2,
\end{cases}
\]
and
\[
a_n\big(\Sc_p^N \hookrightarrow \Sc_q^N\big)
\gtrsim_{p,q}
\begin{cases}
1       
           &:\,  1 \leq n \leq (1-\cHPV)N^2,\\
\sqrt{\frac{N^2-n+1}{N^2}}^{\frac{1/p-1/q}{1/p-1/2}}
          &:\, (1-\cHPV)N^2 \leq n\leq N^2 - \cHPV N^{3-2/p}+1,\\
N^{1/q-1/p} 
          &:\, N^2 - \cHPV N^{3-2/p} +1\leq n \leq N^2.
\end{cases}
\]
	\end{enumerate}
\item
\label{item-appr-quasi-low}
If $0<p\leq 1$, $2\leq q\leq \infty$, then
\[
a_n\big(\Sc_p^N\hookrightarrow \Sc_q^N\big) = a_n\big(\Sc_1^N\hookrightarrow \Sc_q^N\big)
\asymp_q 
\begin{cases}
\min\Bigl\{ 1, \frac{N^{1/2+1/q}}{n^{1/2}}\Bigr\}
     &:\, 1 \leq n\leq (1-\cHPV)N^2, \\
N^{-3/2+1/q} \sqrt{N^2-n+1}
    &:\, (1-\cHPV )N^2  \leq n\leq N^2 -  \cHPV N+1,\\
N^{1/q-1}
    &:\,  N^2 -\cHPV N +1\leq n \leq N^2.
\end{cases}
\]
\item
\label{item-appr-quasi-up}
If $0<p\leq 1\leq q\leq 2$, then $a_n\big(\Sc_p^N\hookrightarrow \Sc_q^N\big) = a_n\big(\Sc_1^N\hookrightarrow \Sc_q^N\big)$ and so, by item~\ref{item-appr-two-triangles-b} applied with $p=1$, we get the estimates 
\[
a_n\big(\Sc_1^N \hookrightarrow \Sc_q^N\big)
\lesssim
\begin{cases}
	1       
	&:\,  1 \leq n \leq N^2 - c(1) N^{3-2/q} +1,\\
	N^{-3/2 + 1/q} \sqrt{N^2-n+1}
	&:\, N^2 - c(1) N^{3-2/q} +1\leq n\leq N^2 -  \cHPV N+1,\\
	N^{1/q-1} 
	&:\, N^2 - \cHPV N +1\leq n \leq N^2,
\end{cases}
\]
and
\[
a_n\big(\Sc_1^N \hookrightarrow \Sc_q^N\big)
\gtrsim_{q}
\begin{cases}
	1       
	&:\,  1 \leq n \leq (1-\cHPV)N^2,\\
	\Bigl(\frac{N^2-n+1}{N^2}\Bigr)^{1-1/q}
	&:\, (1-\cHPV)N^2 \leq n\leq N^2 - \cHPV N+1,\\
	N^{1/q-1} 
	&:\, N^2 - \cHPV N +1\leq n \leq N^2.
\end{cases}
\]
\item
\label{item-appr-quasi-1}
If $0<p \leq q\leq 1$, then
\[
a_n\big(\Sc_p^N\hookrightarrow \Sc_q^N\big) = 1.\]
\end{enumerate}
\end{thmalpha}

Finally, for the convenience of the reader and future reference, we summarize all the known results for Kolmogorov numbers
(including those which follow directly by duality from respective estimates for Gelfand numbers contained, e.g., in \cite{HPV2020}).
Our contribution is the case $0<p<1$ and the estimates in the intermediate ranges implied by Theorem~\ref{thm:main-Gelfand} above. 


\begin{figure}[h] 
		\begin{subfigure}{0.45\textwidth}
			\centering
		\begin{tikzpicture}[scale=2]
		\draw [fill=MyFillColor] (0,0) -- (3,3) -- (0,3);
		\draw [fill=MyFillColor] (1,1) -- (2,1) -- (2,2);
		\draw [fill=MyFillColor] (2,1) -- (3,1) -- (3,3) -- (2,2);
		\draw [pattern=MyPattern, pattern color=MyLinesColor] (0,0) -- (1,0) -- (1,1);
		\draw [fill=MyFillColor] (1,0) rectangle (2,1);
		\draw [pattern=MyPatternOther, pattern color=MyLinesColor] (2,0) rectangle (3,1);
		\draw (0,0) -- (0,3);
		\draw (0,0) -- (3,0);
		\draw (0,3) -- (3,3);
		\draw (3,0) -- (3,3);
		\draw (0,0) -- (3,3);
		\draw (1,0) -- (1,1);
		\draw (2,0) -- (2,2);
		\draw [loosely dotted] (0,1) -- (1,1);
		\draw (1,1) -- (3,1);
		\draw [loosely dotted] (0,2) -- (2,2);
		\node [below] at (-0.05,0) {$q=p=\infty$};
		\node [below] at (1,0) {$p=2$};
		\node [below] at (2,0) {$p=1$};
		\node [below] at (3,0) {$p=0$};		
		\node [left] at (0,1) {$q=2$};
		\node [left] at (0,2) {$q=1$};
		\node [left] at (0,3)  {$q=0$};
					\node [] at (1.8,0.2) {};
					\node [] at (2.8,0.2) {};
		\end{tikzpicture}
		\caption{{Gelfand numbers}}
		 \label{subfig:Gelfand}
\end{subfigure}
\qquad
\begin{subfigure}{0.45\textwidth}
	\centering
		\begin{tikzpicture}[scale=2]
			\draw [fill=MyFillColor] (0,0) -- (2,2) -- (0,2);
			\draw [fill=MyFillColor] (0,0) -- (1,0) -- (1,1);
			\draw [fill=MyFillColor] (1,0) rectangle (2,1);
			\draw [fill=MyFillColor] (2,0) rectangle (3,1);
			\draw [fill=MyFillColor] (2,2) -- (3,2) -- (3,3);			
			\draw [pattern=MyPattern, pattern color=MyLinesColor] (1,1) -- (2,1) -- (2,2);
			\draw [pattern=MyPattern, pattern color=MyLinesColor] (2,1) rectangle (3,2);
			\draw (0,0) -- (0,3);
			\draw (0,0) -- (3,0);
			\draw (0,3) -- (3,3);
			\draw (3,0) -- (3,3);
			\draw (0,0) -- (3,3);
			\draw (1,0) -- (1,1);
			\draw (2,0) -- (2,2);
			\draw [loosely dotted] (0,1) -- (1,1);
			\draw (1,1) -- (3,1);
			\draw (0,2) -- (3,2);
			\node [below] at (-0.05,0) {$q=p=\infty$};
			\node [below] at (1,0) {$p=2$};
			\node [below] at (2,0) {$p=1$};
			\node [below] at (3,0) {$p=0$};		
			\node [left] at (0,1) {$q=2$};
			\node [left] at (0,2) {$q=1$};
			\node [left] at (0,3)  {$q=0$};
			\node [] at (0.8,2.2) {\ref{item-Kolmogorov-upper-triangle-q-quasi}.};
			\node [] at (0.8,1.2) {\ref{item-Kolmogorov-upper-triangle}.};
			\node [] at (1.8,0.2) {\ref{item-Kolmogorov-square}.};
			\node [] at (0.8,0.2) {\ref{item-Kolmogorov-two-triangles-a}.};
			\node [] at (1.8,1.2) {\ref{item-Kolmogorov-two-triangles-b}.};
			\node [] at (2.8,0.2) {\ref{item-Kolmogorov-quasi-low}.};
			\node [] at (2.8,1.2) {\ref{item-Kolmogorov-quasi-up}.};		
			\node [] at (2.8,2.2) {\ref{item-Kolmogorov-quasi-1}.};			
		\end{tikzpicture}
		\caption{Kolmogorov numbers}
		\label{subfig:Kolmogorov}
			\end{subfigure}
	\captionsetup{singlelinecheck=off}
	\caption[foo]{Summary of the results for Gelfand and  Kolmogorov numbers.
			Gray regions correspond to exact asymptotics, northeast (resp.\ northwest) lined regions to exact asymptotics for small (resp.\ intermediate) and large $n$, but nonmatching upper and lower bounds in the intermediate (resp.\ small) range of $n$.
		More precisely, for Kolmogorov numbers:
		\begin{itemize}
			\item[] \ref{item-Kolmogorov-upper-triangle-q-quasi} -- upper bound known, no lower bound;
			\item[] \ref{item-Kolmogorov-upper-triangle} -- $d_n(\Sc_p^N\hookrightarrow \Sc_q^N) \asymp_{p,q} \max\bigl\{ 1, \frac{N^2-n+1}{N}\bigr\}^{1/q-1/p}$;
			\item[] \ref{item-Kolmogorov-square}, \ref{item-Kolmogorov-two-triangles-a} -- exact asymptotics known;	
			\item[]	\ref{item-Kolmogorov-two-triangles-b} -- nonmatching upper and lower bounds in the intermediate range;
			\item[]	\ref{item-Kolmogorov-quasi-low} -- no dependence on $p$: $d_n(\Sc_p^N\hookrightarrow \Sc_q^N) = d_n(\Sc_1^N\hookrightarrow \Sc_q^N)$, exact asymptotics known;
			\item[] \ref{item-Kolmogorov-quasi-up} --  no dependence on $p$: $d_n(\Sc_p^N\hookrightarrow \Sc_q^N) = d_n(\Sc_1^N\hookrightarrow \Sc_q^N)$, nonmatching upper and lower bounds in the intermediate range as in \ref{item-Kolmogorov-two-triangles-b};
			\item[] \ref{item-Kolmogorov-quasi-1} -- $d_n(\Sc_p^N\hookrightarrow \Sc_q^N) =1$.
		\end{itemize}
	}   
	\label{fig:Kolmogorov-results} 
\end{figure}
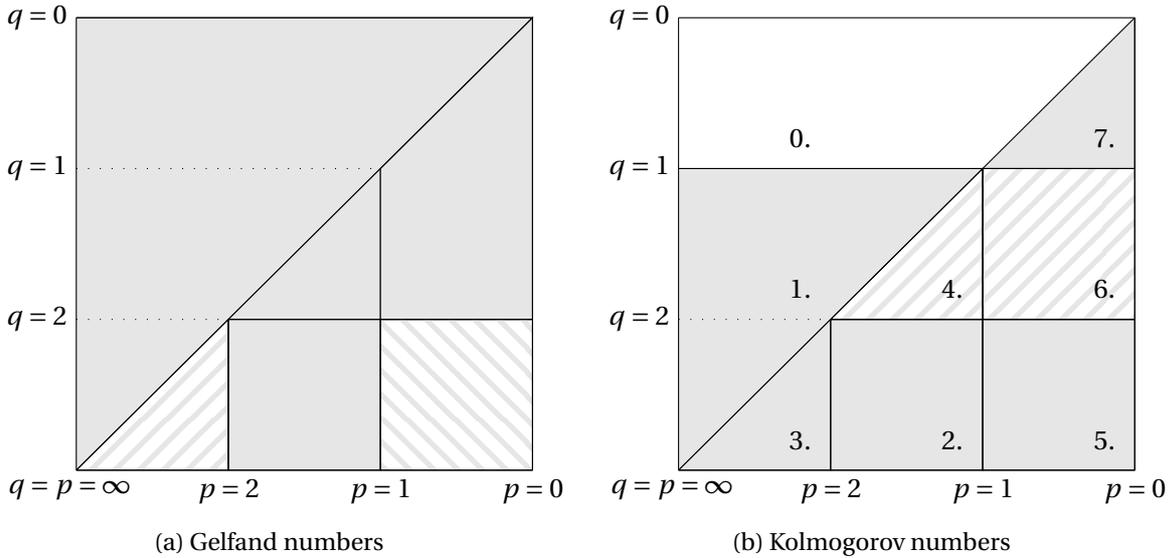

 
\begin{thmalpha}[Results for Kolmogorov numbers]\label{thm: main-Kolmogorov}
	Let $0<p,q\leq \infty$ and assume that $n,N\in\N$ with $1\leq n \leq N^2$.
	Let $\cHPV \in(0,1)$ be the universal constant from \cite[Lemma 2.5]{HPV2020}.
	Then the following estimates hold (see Figure~\ref{fig:Kolmogorov-results}).
	\begin{enumerate}
		\setcounter{enumi}{-1} 
		\item
		\label{item-Kolmogorov-upper-triangle-q-quasi}
		If $0< q <1$, $q\leq p \leq \infty$, then
		\[
		d_n\big(\Sc_p^N\hookrightarrow \Sc_q^N\big)
		\lesssim_{p,q} 
		\max\Bigl\{ 1, \frac{N^2-n+1}{N}\Bigr\}^{1/q-1/p}.
		\]
		Moreover, there exists a constant $c(p,q)\in(0,2]$ such that
        \[
			d_{\lfloor c(p,q) N^2\rfloor+1}\big(\Sc_p^{2N} \hookrightarrow \Sc_q^{2N}\big) 
			\gtrsim_{p,q} N^{1/q-1/p}.
		\]
		\item
		\label{item-Kolmogorov-upper-triangle}
		If $1\leq q\leq p \leq \infty$, then
		\[
		d_n\big(\Sc_p^N\hookrightarrow \Sc_q^N\big)
		\asymp_{p,q} 
		\max\Bigl\{ 1, \frac{N^2-n+1}{N}\Bigr\}^{1/q-1/p}.
		\]
		\item
		\label{item-Kolmogorov-square}
		If $1\leq p\leq 2\leq q\leq \infty$, then
		\begin{align*}
			d_n\big(\Sc_p^N\hookrightarrow \Sc_q^N\big)
			\asymp_{p,q} 
			\begin{cases}
				\min\Bigl\{ 1, \frac{N^{1/2+1/q}}{n^{1/2}}\Bigr\} 
				&:\, 1 \leq n\leq (1-\cHPV)N^2, \\
				N^{-3/2 + 1/q} \sqrt{N^2-n+1}
				&:\, (1-\cHPV)N^2 \leq n \leq N^2 -  \cHPV N^{3-2/p}+1,\\
				N^{1/q-1/p}
				&:\,  N^2 -\cHPV N^{3-2/p} +1\leq n \leq N^2.
			\end{cases}
		\end{align*}
		\item 
		\label{item-Kolmogorov-two-triangles-a}
			If $2\leq p\leq q\leq \infty$, then 
			\[
			d_n\big(\Sc_p^N \hookrightarrow \Sc_q^N\big)
			\asymp_{p,q}
\min\Biggl\{1,\frac{N^{1/2+1/q}}{n^{1/2}}\Biggr\}^{\frac{1/p-1/q}{1/2-1/q}}.
			\]
		\item
		\label{item-Kolmogorov-two-triangles-b}
			If $1\leq p\leq q\leq 2$, then  $d_n\big(\Sc_p^N\hookrightarrow \Sc_q^N\big)  = c_n\big(\Sc_{q^*}^N\hookrightarrow \Sc_{p^*}^N\big)$ and so we get the estimates 
			\[
			d_n\big(\Sc_p^N \hookrightarrow \Sc_q^N\big)
			\lesssim_{p}
			\begin{cases}
				1       
				&:\,  1 \leq n \leq N^2 - c(p) N^{3-2/q} +1,\\
				N^{-3/2 + 1/q} \sqrt{N^2-n+1}
				&:\, N^2 - c(p) N^{3-2/q} +1\leq n\leq N^2 -  \cHPV N^{3-2/p}+1,\\
				N^{1/q-1/p} 
				&:\, N^2 - \cHPV N^{3-2/p} +1\leq n \leq N^2,
			\end{cases}
			\]
			and
			\[
			d_n\big(\Sc_p^N \hookrightarrow \Sc_q^N\big)
			\gtrsim_{p,q}
			\begin{cases}
				1       
				&:\,  1 \leq n \leq (1-\cHPV )N^2,\\
				\sqrt{\frac{N^2-n+1}{N^2}}^{\frac{1/p-1/q}{1/p-1/2}}
				&:\, (1-\cHPV )N^2 \leq n\leq N^2 - \cHPV  N^{3-2/p}+1,\\
				N^{1/q-1/p} 
				&:\, N^2 - \cHPV  N^{3-2/p} +1\leq n \leq N^2.
			\end{cases}
			\]
		\item
		\label{item-Kolmogorov-quasi-low}
		If $0<p\leq 1$, $2\leq q\leq \infty$, then
		\[
		d_n\big(\Sc_p^N\hookrightarrow \Sc_q^N\big) = d_n\big(\Sc_1^N\hookrightarrow \Sc_q^N\big)
		\asymp_q 
		\begin{cases}
			\min\Bigl\{ 1, \frac{N^{1/2+1/q}}{n^{1/2}}\Bigr\}
			&:\, 1 \leq n\leq (1-\cHPV)N^2, \\
			N^{-3/2+1/q} \sqrt{N^2-n+1}
			&:\, (1-\cHPV )N^2  \leq n\leq N^2 -  \cHPV N+1,\\
			N^{1/q-1}
			&:\,  N^2 -\cHPV N +1\leq n \leq N^2.
		\end{cases}
		\]
		\item
		\label{item-Kolmogorov-quasi-up}
		If $0<p\leq 1\leq q\leq 2$, then $d_n\big(\Sc_p^N\hookrightarrow \Sc_q^N\big) = d_n\big(\Sc_1^N\hookrightarrow \Sc_q^N\big)$ and so, by item~\ref{item-Kolmogorov-two-triangles-b} applied with $p=1$, we get the estimates 
		\[
		d_n\big(\Sc_1^N \hookrightarrow \Sc_q^N\big)
		\lesssim
		\begin{cases}
			1       
			&:\,  1 \leq n \leq N^2 - c(1) N^{3-2/q} +1,\\
			N^{-3/2 + 1/q} \sqrt{N^2-n+1}
			&:\, N^2 - c(1) N^{3-2/q} +1\leq n\leq N^2 -  \cHPV N+1,\\
			N^{1/q-1} 
			&:\, N^2 - \cHPV N +1\leq n \leq N^2,
		\end{cases}
		\]
		and
		\[
		d_n\big(\Sc_1^N \hookrightarrow \Sc_q^N\big)
		\gtrsim_{q}
		\begin{cases}
			1       
			&:\,  1 \leq n \leq (1-\cHPV)N^2,\\
			\Bigl(\frac{N^2-n+1}{N^2}\Bigr)^{1-1/q}
			&:\, (1-\cHPV)N^2 \leq n\leq N^2 - \cHPV N+1,\\
			N^{1/q-1} 
			&:\, N^2 - \cHPV N +1\leq n \leq N^2.
		\end{cases}
		\]		
		\item
		\label{item-Kolmogorov-quasi-1}
		If $0<p \leq q\leq 1$, then
		\[
		d_n\big(\Sc_p^N\hookrightarrow \Sc_q^N\big) = 1.\]
	\end{enumerate}
\end{thmalpha}

\subsection{Missing cases and conjectures}

Below we list the cases where Theorems~\ref{thm:main-Gelfand}, \ref{thm: main-approximation}, and \ref{thm: main-Kolmogorov}
do not describe the exact asymptotics of approximation, Gelfand, and Kolmogorov numbers,
and  pose conjectures about the behavior of the $s$-numbers in those cases.
In particular, we believe that for $0\leq p, q \leq \infty$ it is true that
\[
a_n\big(\Sc_p^N \hookrightarrow \Sc_q^N\big) \asymp_{p,q} \max\Bigl\{c_n\big(\Sc_p^N \hookrightarrow \Sc_q^N\big), d_n\big(\Sc_p^N \hookrightarrow \Sc_q^N\big)\Bigr\}.
\]

The first conjecture appears already in \cite{GKS1987} and concerns the logarithms in the upper bound for approximation numbers in the case corresponding to region~\ref*{item-appr-square} in Figure~\ref{fig:approximation results}.

\begin{conjecture}
	\label{conj:log}
	If $ 1< p\leq q < \infty$ and $n,N\in \N$ with $(1-\cHPV) N^2 \leq n \leq N^2-\cHPV N^{\critmax_{p,q}} + 1$, then
	\begin{equation*}
		a_n\big(\Sc_p^N \hookrightarrow \Sc_q^N\big)
		 \lesssim_{p,q} N^{\critmax_{p,q}/2-2} \sqrt{N^2-n+1}.
	\end{equation*}
\end{conjecture}

The next conjecture is about the behavior of Gelfand and approximation numbers in the triangle on the lower left-hand side of Figure~\ref{subfig:Gelfand}.
By duality and Lemma~\ref{lem:approximation-Kolmogorov-p-quasi}, it can be also equivalently stated in terms of the behavior of Kolmogorov numbers in regions~\ref*{item-Kolmogorov-two-triangles-b} and~\ref*{item-Kolmogorov-quasi-up} in Figure ~\ref{subfig:Kolmogorov}.

\begin{conjecture}
	\label{conj:HiMi}
If $ 2 \leq p\leq q \leq \infty$ and $n,N\in \N$ with $(1-\cHPV) N^2 \leq n \leq N^2 - \cHPV N^{1+2/q} + 1$, then
\begin{equation*}
 		c_n\big(\Sc_p^N \hookrightarrow \Sc_q^N\big) 
 		\leq a_n\big(\Sc_p^N \hookrightarrow \Sc_q^N\big) 
 		\lesssim_{p,q} \sqrt{\frac{N^2-n+1}{N^2}}^{\frac{1/p-1/q}{1/2-1/q}}.
\end{equation*}
\end{conjecture}

The second unsettled case for Gelfand numbers corresponds to the square region on the lower right-hand side of Figure~\ref{subfig:Gelfand}.

\begin{conjecture}
	\label{conj:Gelfand-quasi}
	If $ 0< p \leq  1$, $2\leq q \leq \infty$, and $n,N\in \N$ with $1\leq n \leq (1-\cHPV) N^2$, then
	\begin{equation*}
		c_n\big(\Sc_p^N \hookrightarrow \Sc_q^N\big)
		 \gtrsim_{p,q} \min\Big\{1,\frac{N}{n} \Big\}^{1/p-1/2}.
	\end{equation*}
\end{conjecture}

The final conjecture is about Kolmogorov numbers in the case when the codomain is a quasi-Banach space (region~\ref*{item-Kolmogorov-upper-triangle-q-quasi} in Figure~\ref{subfig:Kolmogorov}).

\begin{conjecture}
	\label{conj:Kolmogorov}
	If $ 0 < p \leq  \infty$, $0< q \leq 1$, $q\leq p$, and $n,N\in \N$ with $1\leq n \leq N^2$, then
	\begin{equation*}
		d_n\big(\Sc_p^N \hookrightarrow \Sc_q^N\big) \gtrsim_{p,q} \max\Bigl\{ 1, \frac{N^2-n+1}{N}\Bigr\}^{1/q-1/p}.
	\end{equation*}
\end{conjecture}


\subsection{Relation to widths and recovery problems}
The intimate relation between compressed sensing and geometric functional analysis, in particular asymptotic geometric analysis, is well understood and both fields have gained from cross-fertilization (see, e.g., \cite{CGLP2012}). On the non-commutative front it is nicely demonstrated in \cite{CDK2015} how the relation between the geometry of Schatten classes and the notion of Gelfand numbers (and widths) and low-rank matrix recovery goes beyond nuclear norm minimization. The Gelfand widths, which in our setting of identity mappings coincide with the Gelfand numbers up to a shift in the index, are equivalent to the worst-case recovery error under optimal information and recovery schemes (see \cite[Lemma 2.5.3]{CGLP2012} and \cite[Theorem 5.5]{CDK2015}). More precisely, assume we wish to recover a matrix $X$ from the unit ball $B_p^N$ of the Schatten class $\Sc_p^N$ from $m$ linear measurements (pieces of information) of $X$ provided by
the information mapping $\mathcal A:\R^{N\times N}\to\R^m$,
\[
{\mathcal A}(X):=\big(\langle A_1,X\rangle,\dots,\langle A_m,X\rangle\big)\in\R^m.
\]
The pursuit for the optimal pair of information mapping ${\mathcal A}:\R^{N\times N}\to \R^m$ and recovery mapping $\Delta:\R^m\to \R^{N\times N}$
is expressed in terms of the quantity
\[
E_m\big(B_p^N,\Sc_q^N\big):=\inf_{\substack{{\mathcal A}\,:\,\R^{N\times N}\to \R^m\\ \Delta\,:\,\R^m\to \R^{N\times N}}}\,\sup_{X\in B_p^N}\|X-\Delta({\mathcal A}(X))\|_{\Sc_q},
\]
measuring the worst-case error of the optimal information--recovery scheme. Using their already mentioned relation to Gelfand widths/ numbers (see \cite[Section 2.5]{CGLP2012} and \cite[Section 10]{FR2013} for details), one obtains that whenever $0<p\le 1$ and $p<q\le 2$ and $1\le m\le N^2$, 
\[
E_m\big(B_p^N,\Sc_q^N\big)\asymp_{p,q}\min\Bigl\{1,\frac{N}{m}\Bigr\}^{1/p-1/q}.
\]
We refer the reader to \cite[Theorem 5.5]{CDK2015} for details and to \cite{HPV2017} for an alternative proof of the lower bound via entropy numbers and Carl's inequality, which was obtained there in the extended regime $0<p\le q\le \infty$. For general background on compressed sensing, we refer the reader to \cite{CGLP2012,DDEYK2012,FR2013}.

\subsection{Organization of the article}

The remainder of this paper is organized as follows.
In Section~\ref{sec:prelim} we present the preliminaries (including notation,  basic  notions  and  background  on  $s$-numbers and  Schatten classes) and recall a variety of results needed in the proofs of Theorems~\ref{thm:main-Gelfand}, \ref{thm: main-approximation}, \ref{thm: main-Kolmogorov}. In Section~\ref{sec:Gelfand} we prove the estimates for Gelfand numbers from Theorem~\ref{thm:main-Gelfand} (and give an elementary proof of results of Ch\'avez-Dom\'inguez and Kutzarova~\cite{CDK2015}). The rest of the paper contains the proofs of the results for approximation and Kolmogorov numbers from Theorems~\ref{thm: main-approximation} and \ref{thm: main-Kolmogorov}, respectively.
In Section~\ref{sec:approximation-upper-triangle} we consider the case $0<q\leq p \leq \infty$. In Section~\ref{sec:approximation-square} we give the proofs of the results for approximation numbers in the case $1\leq p \leq 2\leq q \leq \infty$, while Section~\ref{sec:approximation-two-triangles} covers the cases $1\leq p \leq 2$ and $2\leq p \leq q \leq \infty$. Finally, in Section~\ref{sec:approximation-p-quasi} we prove the estimates for approximation and Kolmogorov numbers in the case $0< p <1$, i.e., when the domain space is a quasi-Banach space.

\section{Preliminaries \& mathematical machinery} \label{sec:prelim}

\subsection{Notation}

For $0< p\leq \infty$, we denote by $\ell_p^N$ the space $\R^N$ equipped with the (quasi-)norm 
\[
\big\|(x_i)_{i=1}^N\big\|_p :=
\begin{cases}
\Big(\sum\limits_{i=1}^N|x_i|^p\Big)^{1/p} & :\  0< p < \infty,\\
\max\limits_{1\leq i \leq N}|x_i| 
& :\  p=\infty.
\end{cases}
\]

Given two quasi-Banach spaces $X$ and $Y$, we shall denote by $\mathscr L(X,Y)$ the space of bounded linear operators between $X$ and $Y$ equipped with the operator quasi-norm. The closed unit ball of a Banach space $X$ will be denoted by $B_X$. For a set $A\subset\R^n$, we denote by $\conv(A)$ the convex hull of $A$, i.e.,
\[
\conv(A):=\Bigg\{ \sum_{i=1}^m \lambda_i x_i\,:\, m\in\N,\,x_1\dots,x_m\in\R^n, \lambda_1,\dots,\lambda_m \geq 0,\,\,\text{and}\,\,\sum_{i=1}^m\lambda_i=1\Bigg \}.
\] 
For sequences $(a_n)_{n\in\N}$ and $(b_n)_{n\in\N}$ of positive real numbers, we shall write $a_n \lesssim b_n$ if there is an absolute constant $c\in(0,\infty)$ such that $a_n \leq c b_n$ for all $n\in\N$. Similarly, we define $a_n \gtrsim b_n$ and write $a_n \asymp b_n$ if both $a_n\lesssim b_n$ and $a_n \gtrsim b_n$. If the constants depend on some parameter $r$, then we express this by writing $\lesssim_r$, $\gtrsim_r$, or $\asymp_r$.

\subsection{The $s$-numbers and their properties}

An axiomatic approach to the study of $s$-numbers goes back to Pietsch \cite{Pietsch2} and several (quantitative) aspects concerning those numbers and their relations to one another can also be found in the monographs \cite{CS1990} and \cite{K1986}.
Let $X,Y, X_0, Y_0$ be quasi-Banach spaces 
and let $p\in(0,1]$ such that $Y$ is a $p$-Banach space.
 A map $s$ which assigns to any bounded linear operator $T$ between two quasi-Banach spaces 
 a sequence $(s_n(T))_{n\in\N}$
 is called an $s$-function, if:
\begin{enumerate}
\item[($S_1$)] $\|T\| = s_1(T) \geq s_2(T) \geq \dots \geq 0$ for all $T\in\mathscr L(X,Y)$;
\item[($S_2$)] $s_{m+n-1}^p(S+T) \leq s_m^p(S) + s_n^p(T)$ for all $S,T\in \mathscr L(X,Y)$ and $m,n\in\N$;
\item[($S_3$)] $s_n(STU) \leq \|S\|s_n(T)\|U\|$ for all $U\in \mathscr L(X_0,X)$, $T\in\mathscr L(X,Y)$, $U\in\mathscr L(Y,Y_0)$, and $n\in\N$;
\item[($S_4$)] $s_n(T) = 0$ for all $T\in \mathscr L(X,Y)$ and $n\in\N$ with $\rank T <n$;
\item[($S_5$)] $s_n(\id\colon \ell_2^n \to \ell_2^n) =1$ for all $n\in \N$.
\end{enumerate}
Note that $p=1$ in ($S_2$) if $Y$ is a Banach space.
 We call $s_n(T)$ the $n-$th $s$-number of $T$.
 An $s$-function is called multiplicative if
 \[
 s_{m+n-1}(ST) \leq s_m(S) s_n(T) \quad  \text{ for all } T\in \mathscr L(X,Y),\ S\in \mathscr L(Y,Y_0), \text{ and } m,n\in\N.
 \]
 
The sequences $(a_n)_{n\in\N},(c_n)_{n\in\N}$, and $(d_n)_{n\in\N}$ of approximation, Gelfand, and Kolmogorov numbers defined in Section \ref{sec:agk numbers for schatten} are all sequences of multiplicative
 $s$-numbers (this follows
directly from their definitions, cf.~\cite{MR2432104}).
 In particular, for all $n\in\N$, we have the relations
\[
a_n(T) \geq c_n(T) \quad\text{and}\quad a_n(T)\geq d_n(T)
\] 
with equality for operators $T$ between Hilbert spaces; in this Hilbert space setting and for compact $T$, it is a consequence of the spectral theorem that those $s$-numbers coincide with the sequence of singular values of $T$.

Moreover, if $T$ is a compact operator between Banach spaces, then we have the duality relations $a_n(T)=a_n(T^*)$ and $c_n(T)=d_n(T^*)$ for all $n\in\N$, where $T^*$ denotes the dual operator.

If $T$ is an isomorphism between $m$-dimensional spaces, then $c_m(T)=1/\|T^{-1}\|$.
We also have the following simple lemma.

\begin{lemma}
	\label{lem:approximation-Kolmogorov-id-1}
Let $X$ be an $m$-dimensional quasi-Banach space 
and let $\id_X$ denote the identity operator from $X$ to $X$.
Then $a_n(\id_X) = c_n(\id_X)  = d_n(\id_X) = 1$ for $n\in\N$ with $n\leq m$.
\end{lemma}

\begin{proof}
By the aforementioned properties,
\[
1 = a_1(\id_X) \geq a_n(\id_X) \geq  c_n(\id_X) \geq c_m(\id_X)  =1.
\]
Thus it remains to prove the statement for the Kolmogorov numbers
(and, by duality, we may focus on the case when $X$ is a quasi-Banach space).

Take any $\varepsilon>0$ and any subspace $E\subset X$ with $\dim E < m$.
Then $E\neq X$ and by the Riesz lemma (see below) there exists $x_1\in X\setminus E$, such that $\|x_1\| = 1$ and $\|x_1 - y\| \geq 1/(1+\varepsilon)$ for all $y\in E$.
Thus
\[
\sup_{x\in B_X}\inf_{y\in E}\|x-y\| \geq \inf_{y\in E} \|x_1 - y\| \geq \frac{1}{1+\varepsilon}.
\]
Since $E$ and $\varepsilon>0$ were arbitrary it follows that $d_n(\id_X) \geq 1$.
\end{proof}

For completeness we include a short proof of the Riesz lemma for quasi-normed spaces.

\begin{lemma}[Riesz's lemma]
	Let $X$ be a quasi-normed space and $E$ a proper closed subspace of $X$.
	For every $\varepsilon>0$ there exists $x_1\in X\setminus E$,
	such that $\|x_1\| = 1$ and $\|x_1 - y\| \geq 1/(1+\varepsilon)$ for all $y\in E$.
\end{lemma}

\begin{proof}	
	Take any $x_0\in X\setminus E$.
	Denote $R\coloneqq \inf_{y\in E} \|x_0 -y\| >0$ and let $y_0\in E$ be such that $\|x_0-y_0\| \leq (1+\varepsilon)R$.
	If we define $x_1 \coloneqq \frac{x_0-y_0}{\|x_0-y_0\|}$, then $x_1\notin E$, $\|x_1\|=1$, and
	\begin{align*}
		\inf_{y\in E} \|x_1 - y \| = \frac{1}{\|x_0-y_0\|}	\inf_{y\in E} \bigl\|x_0-y_0 - \|x_0-y_0\|y \bigr\| = \frac{1}{\|x_0-y_0\|}	\inf_{z\in E} \|x_0-z\| \geq \frac{1}{1+\varepsilon}.
	\end{align*}	
This ends the proof.
		\end{proof}

\subsection{Some general estimates for sequences of $s$-numbers}

We will need some (probabilistic) estimates for and relations between $s$-numbers, which we shall collect here. Below  $d(\cdot,\cdot)$ stands for the Banach--Mazur distance, i.e., if $X$ and $Y$ are isomorphic quasi-Banach spaces, then
\[
d(X,Y):= \inf\big\{\|T\|\|T^{-1}\|\,:\, T\in\mathscr L(X,Y)\,\,\text{isomorphism} \big \}.
\]
The first result concerns lower bounds for the sequence of Gelfand numbers.

\begin{lemma}[{{\cite[Propositon~2.4]{GKS1987}}}]
	\label{lem:GKS-FLM}
	Let $m\in\N$ and suppose that $X$, $Y$ are $m$-dimensional quasi-Banach spaces, $T\in \mathscr{L}(X,Y)$ is an invertible linear operator, 
	and that $T^{-1}:Y\to X$ allows for a decomposition $T^{-1} = BA $ for some $A\in \mathscr L(Y,\ell_2^m)$ and $B\in \mathscr L(\ell_2^m,X)$. 
	Then, for all $n\leq m$,
	\[
	c_{n}\big(T\colon X \to Y\big) 
	\geq 
	\frac{1}{\|A\| \|B\|} \inf\Big\{ d(E, \ell_2^{m-n+1})\, :\, E\subset Y,\ \dim E = m-n+1\Big\}.
	\]
\end{lemma}

To make it clear that the above lemma can also be used in the quasi-Banach setting (and thus for the sake of completeness) we provide a proof below.

\begin{figure}[h] 
\centering
	\begin{tikzcd}
		& \ell_2^m \arrow[dl,swap,"B"] \\
		F_0\subset X && Y \supset T(F_0)  \arrow[ll,swap,"T^{-1}"] \arrow[ul,swap,"A"] 
	\end{tikzcd}
    \caption{Decomposition $T^{-1} = BA$.}
	\label{fig:fact-1} 
\end{figure}
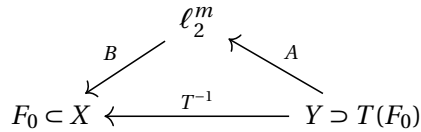

\begin{proof}[Proof of Lemma~\ref{lem:GKS-FLM}]
        Let $\lambda >0$ be such that
        \[
        \lambda \|A \| = \inf\big\{ d(E, \ell_2^{m-n+1}) \,:\, E\subset Y,\ \dim E = m-n+1 \big\}.
        \]
Let $F\subset X$ be any subspace of $X$ with $\codim F < n$, i.e., $\dim F \geq m-n +1$.
        Let $F_0 \subset F$ be any subspace of $X$ with $\dim F_0 = m-n+1$.
        We claim that there exists a vector $x_0\in F_0$ such that
        \[
        \|B^{-1} x_0 \| = 1 \quad \text{ and } \quad  \|T x_0\| \geq \lambda.
        \]
        Indeed, otherwise  for all $x\in F_0$ such that $\|B^{-1} x \| = 1$ we would have $ \|T x\| < \lambda$, and consequently, by compactness,

        \[
        \| A^{-1} |_{B^{-1}(F_0)} \| < \lambda.
        \]
        From this, we obtain 
        \begin{align*}
                \inf\big\{ d(E, \ell_2^{m-n+1}) : E\subset Y,\ \dim E = m-n+1\big\} 
                 \leq d\big(T(F_0), \ell_2^{m-n+1}\big)
                 \leq \| A^{-1} |_{B^{-1}(F_0)} \| \|A |_{T(F_0)}\| < \lambda \|A\|,
        \end{align*}
        which contradicts the definition of $\lambda$. Therefore, by our choice of $x_0\in F_0$ and the definition of $\lambda$,
        \begin{align*}
                \| T |_F \|
                &\geq \| T |_{F_0}\|  \geq \frac{\| T x_0\|}{ \|x_0\|} \geq \frac{\lambda}{ \|B\| \|B^{-1} x_0\|} \\
                &= \frac{\inf\big\{ d(E, \ell_2^{m-n+1}) : E\subset Y,\ \dim E = m-n+1\big\}} {\|A \| \|B\|}.
        \end{align*}
        The result follows by taking the infimum over all subspaces $F\subset X$ with $\codim F < n$.
\end{proof}

The next result provides a lower bound on the Kolmogorov numbers.

\begin{lemma}[{{\cite[Propositon~2.3]{GKS1987}}}]
	\label{lem:GKS-Kolmogorov-lower-bound}
	Let $m\in\N$ and suppose that $X$, $Y$ are $m$-dimensional Banach spaces and $T\in\mathscr{L}(X,Y)$.
	Moreover, consider linearly independent $y_1, \ldots, y_m\in B_Y$ and $I\in\mathscr L (Y,\ell^m_2)$ such that $I(y_i) = e_i$ for $i=1,\ldots, m$.
	Then, for all $n\leq m$,
	\[
	d_{n}\big(T\colon X \to Y\big) 
	\geq 
	\frac{m - \sqrt{(n-1)m}}{\sqrt{(n-1)m} \|I\colon Y\to \ell^m_2 \| + \sum_{i=1}^m \| I^*(e_i^*)\|}.
	\]
\end{lemma}

\subsection{The Schatten classes}

The singular values $\sigma_1,\dots,\sigma_N$ of a real $N\times N$ matrix $A$ are defined to be the square roots of the eigenvalues of the positive self-adjoint operator $A^*A$, which are simply the eigenvalues of $|A|:=\sqrt{A^*A}$. The singular values are arranged in non-increasing order, that is, $\sigma_1(A) \geq \dots \geq \sigma_N(A)\geq 0$. The singular value decomposition shall be used in the form $A=U\Sigma V^T$, where $U,V\in\R^{N\times N}$ are orthogonal matrices,
and $\Sigma\in\R^{N\times N}$ is a diagonal matrix with $\sigma_1(A),\dots,\sigma_N(A)$ on the diagonal.

For $0<p\le \infty$, the Schatten $p$-class $\Sc_p^N$ is the $N^2$-dimensional space of all $N\times N$ real matrices acting from $\ell_2^N$ to $\ell_2^N$ equipped with the Schatten $p$-norm (which is a quasi-norm for $0<p<1$)
\[
\|A\|_{\Sc_p} := \bigg(\sum_{j=1}^N\sigma_j(A)^p\bigg)^{1/p}.
\]
Let us remark that $\|\cdot\|_{\Sc_1}$ is the nuclear norm, $\| \cdot \|_{\Sc_2}$ the Hilbert--Schmidt norm, and $\|\cdot\|_{\Sc_\infty}$ the operator norm. We denote the unit ball of the Schatten class $\Sc_p^N$ by
\begin{align*}
B_p^N :=\Big\{A\in\R^{N\times N}\,:\,\|A\|_{\Sc_p}\le 1\Bigr\}.
\end{align*}

We shall use the following well-known results about the Schatten classes. The first concerns the operator norms of natural identities, where the behavior is the same as for the commutative sequence spaces $\ell_p$.

\begin{lemma}
	\label{lem:norm}
	For all $0<p,q\leq \infty$, we have 
	\[
	\| \Sc_p^N \hookrightarrow \Sc_q^N \| = \max\big\{ 1, N^{1/q-1/p}\big\}.
	\]
\end{lemma}

The following result shows that, just like for the $\ell_p^n$ balls, the convex hull of the Schatten unit balls $B_p^N$ with $0<p\leq 1$ is the unit ball of the Schatten trace class. The proof is simple and we include it for the sake of completeness. 

\begin{lemma}
	\label{lem:convex hull quasi balls}
	For all $0<p\leq 1$, we have 
	\[
	\conv(B_p^N) = B_1^N.
	\]
\end{lemma}
\begin{proof}
Obviously, due to minimality $\conv(B_p^N)\subset B_1^N$. The other direction follows directly from the singular value decomposition. Indeed, let $A\in B_1^N$ with $\|A\|_{\Sc_1}=1$. Then there exist orthogonal matrices $U,V\in\R^{N\times N}$ and $\Sigma$ such that $A=U\Sigma V^T$, where $\Sigma$ is the diagonal matrix with $s_1(A),\dots,s_N(A)\geq 0$ on the diagonal. Since $A\in B_1^N$, $\sum_{i=1}^Ns_i(A)=1$ and therefore, writing $E_{ii}$, $i\in\{1,\dots,N\}$ for the $N\times N$ matrix with $E_{ii}(k,\ell) = 1$ for $i=k=\ell$ and $0$ otherwise, one has
\[
A = U\Sigma V^T = \sum_{i=1}^N s_i(A) \big[ U E_{ii} V^T\big],
\] 
with $UE_{ii}V^T\in B_p^N$. This completes the proof.
\end{proof}

The next result concerns the optimal Dvoretzky-dimension of Schatten $p$-classes and goes back to Figiel, Lindenstrauss, and Milman and their groundbreaking paper \cite{FLM1977}. 

\begin{lemma}[{{\cite[Example 3.3]{FLM1977}}}]
	\label{lem:Schatten-largest-ell-2-iso}
	The largest dimension $k = k(q)$ such that the space $\Sc^N_q$ contains a subspace $E$ with $\dim E = k$ and $d(E,\ell^k_2) \leq 2$, satisfies 
	\begin{equation*}
		k(q) \asymp
		\begin{cases}
			N^2 &:\, 1\leq  q \leq 2\\
			N^{1+2/q} &:\,  2\leq  q \leq \infty.
		\end{cases}
	\end{equation*}
\end{lemma}

\subsection{Schatten class embeddings and $s$-numbers}

Below we recall and collect some facts about interpolation of $s$-numbers in the context of embeddings of Schatten classes.
More specifically, let us recall that the Gelfand numbers interpolate in the codomain space in the sense that if $0< p, q\leq \infty$, and $q_\theta$ lies between $p$ and $q$, then
\begin{equation}
\label{eq:Gelfand-interpolation}
c_n\big(\Sc^N_{p}\hookrightarrow \Sc^N_{q_\theta}\big) \leq c_n\big(\Sc^N_{p}\hookrightarrow \Sc^N_{q}\big)^{\frac{1/p-1/q_\theta}{1/p-1/q}},
\end{equation}
see, e.g., \cite[Lemma 2.1]{HPV2020}.
Similarly, the Kolmogorov numbers interpolate in the domain space.

\begin{lemma}
	\label{lem:Kolmogorov-interpolation}
	Let $0 < p, q \leq \infty$, let $\theta\in[0,1]$,
	and let $p_\theta$ be the weighted harmonic mean of $p$ and $q$ defined by
	\begin{equation}
		\frac{1}{p_{\theta}} = \frac{\theta}{q} + \frac{1-\theta}{p}.
	\end{equation}
	Assume that $n,N\in \N$ with $1\leq n \leq N^2$.
	Then
	\begin{equation}
		\label{eq:Kolmogorov-interpolation}
		d_n\big(\Sc^N_{p_\theta}\hookrightarrow \Sc^N_{q}\big) \lesssim_{p, p_\theta, q} d_n\big(\Sc^N_{p}\hookrightarrow \Sc^N_{q}\big)^{\frac{1/p_\theta - 1/q}{1/p - 1/q}},
	\end{equation}
\end{lemma}

 \begin{proof}
 	Fix all the parameters appearing in the statement and consider Peetre's $K$-functional defined by
 	\[
 	K(t,A) \coloneqq \inf \bigl\{ \|X\|_{\Sc_p} + t \|Y\|_{\Sc_q} \,:\, A = X+Y, \ X\in\Sc_p^N, \ Y\in\Sc_q^N \bigr\},
 	\quad
 	t\in(0,\infty),\ A\in\Sc_{p_\theta}^N.
 	\]
 	Note that for all $A\in \Sc_{p_\theta}^N$ and $t>0$, 
 	\begin{equation}
 		\label{eq:peetre}
 	\|A\|_{\Sc_{p_\theta}} \gtrsim_{p,p_\theta, q} t^{-\theta} K(t, A),
 	\end{equation}
 	where we used the singular value decomposition and the corresponding $K$-functional estimate for the $\ell_p$-norms (see, e.g., \cite[Theorems 3.11.2 and 5.6.1]{BL_book}). 
 	
 	Now, take any $\varepsilon>0$ and
 	let $E\subset \Sc_q^N$ be a subspace with $\dim E < n$ such that
 	\begin{equation}
 	\label{eq:Kolm-int-p-q-eps}
 	\inf_{B\in E} \| A - B\|_{\Sc_q}
 	 \leq 
 	 (1+\varepsilon)  d_n\big(\Sc^N_{p}\hookrightarrow \Sc^N_{q}\big) \|A\|_{\Sc_p}
 	\end{equation}
 	for all linear operators $A\in\Sc_{p}^N$. Consider any linear operator $A\in\Sc_{p_\theta}^N$.
 	The definition of $p_\theta$ means that $1-\theta = \frac{1/p_\theta-1/q}{1/p-1/q}$, i.e., in order to prove the claimed estimate, it is enough to show that there exists $B_0\in E$ such that
 	\[
 	\| A - B_0\|_{\Sc_q}
 	\leq 
 	(1+\varepsilon)  d_n\big(\Sc^N_{p}\hookrightarrow \Sc^N_{q}\big)^{1-\theta} \|A\|_{\Sc_{p_\theta}}.
 	\] 	
 	To this end, we set $t\coloneqq d_n(\Sc^N_{p}\hookrightarrow \Sc^N_{q})^{-1}$.
 	By~\eqref{eq:peetre}, we can find $X\in\Sc_p^N$, $Y\in\Sc_q^N$ such that $ A = X+Y$ and
 	\begin{equation}
 		\label{eq:int-Kolm-p-X-Y}
 	\|A\|_{\Sc_{p_\theta}} \gtrsim_{p,p_\theta, q} t^{-\theta} \bigl( \|X\|_{\Sc_p} + t \|Y\|_{\Sc_q} \bigr).
 	\end{equation}
 	By \eqref{eq:Kolm-int-p-q-eps}, there exists $B_0\in E$ such that
 	\begin{equation}
 		\label{eq:Kolm-int-p-B_0}
 		\| X - B_0\|_{\Sc_q}
 	\leq 
 	(1+\varepsilon)  d_n\big(\Sc^N_{p}\hookrightarrow \Sc^N_{q}\big) \|X\|_{\Sc_p}.
 	\end{equation} 
    It follows that (recall the choice $t = d_n(\Sc^N_{p}\hookrightarrow \Sc^N_{q})^{-1}$)
    \begin{align*}
    	 \| A - B_0\|_{\Sc_q} 
    	 &=  \| X+ Y - B_0\|_{\Sc_q} \lesssim_q  \| X - B_0\|_{\Sc_q} + \|Y\|_{\Sc_q}\\
      &\leq (1+\varepsilon) t^{-1}  \| X\|_{\Sc_p} + \|Y\|_{\Sc_q} \leq   (1+\varepsilon) t^{-1}  \bigl( \| X\|_{\Sc_p} + t\|Y\|_{\Sc_q}\bigr)\\
      & \lesssim_{p,p_\theta, q}  (1+\varepsilon) t^{-1+\theta}  \| A\|_{\Sc_{p_\theta}} = (1+\varepsilon)  d_n\big(\Sc^N_{p}\hookrightarrow \Sc^N_{q}\big)^{1-\theta} \|A\|_{\Sc_{p_\theta}},
    \end{align*}
    where we first used the triangle inequality (with some constant depending on $q$ if $q\in(0,1)$), then inequality \eqref{eq:Kolm-int-p-B_0} and inequality  \eqref{eq:int-Kolm-p-X-Y}, and finally the definition of $t$.
    This completes the proof.
 \end{proof}

The next lemma  describes the behavior of the Gelfand numbers when the domain space changes. It was used by Hinrichs and Michels in \cite{HM2005} to obtain the lower bound for $c_n\big(\Sc^N_{p}\hookrightarrow \Sc^N_{q}\big)$ in the intermediate range for $2\leq p\leq q$.
For the sake of completeness we recall the proof in the specific setting needed in this paper. 

\begin{lemma}[{{\cite[Proposition 3.1, Proposition 6.1]{MR2223589}}}]
	\label{lem:Gelfand-HiMi-interpolation}
	Let $0 < p, q \leq \infty$, let $\theta\in[0,1]$,
	and let $p_\theta$ be the weighted harmonic mean of $p$ and $q$ defined by
	\begin{equation}\label{eq:interpolation parameters lower bound gelfand}
		\frac{1}{p_{\theta}} = \frac{1-\theta}{q} + \frac{\theta}{p}.
	\end{equation}
	Assume that $n,N\in \N$ with $1\leq n \leq N^2$.
	Then
	\[
	c_n\big(\Sc^N_{p_{\theta}}\hookrightarrow \Sc^N_{q}\big) \geq c_n\big(\Sc^N_{p}\hookrightarrow \Sc^N_{q}\big)^{\frac{1/p_{\theta} - 1/q}{1/p - 1/q}}.
	\]
\end{lemma}

\begin{proof}
	Let $F$ be any subspace of $\Sc^N_{p_{\theta}}$ with $\codim F < n$.
	The definition of $p_\theta$ means that $\theta = \frac{1/p_{\theta}-1/q}{1/p-1/q}$.
	In order to prove the claimed estimate, it is enough to show that there exists a linear operator $A\in F$, $A\neq 0$, such that 
	\[
	\|A\|_{\Sc_{q}}   \geq c_n\big(\Sc^N_{p}\hookrightarrow \Sc^N_{q}\big)^{\theta} \|A \|_{\Sc_{p_{\theta}}}.
	\]
	We may algebraically consider $F$ as a subspace of $\Sc_{p}^N$. Hence, there must exist $A\in F$, $A\neq 0$, such that
	\[
	\|A\|_{\Sc_{q}}   \geq c_n\big(\Sc^N_{p}\hookrightarrow \Sc^N_{q}\big)  \|A \|_{\Sc_{p}}.
	\]
	Then, first using Littlewood's interpolation inequality with \eqref{eq:interpolation parameters lower bound gelfand} (which follows directly from the generalized H\"older inequality) followed by the previous estimate, we obtain
	\[
	\|A \|_{\Sc_{p_{\theta}}} \leq   \|A \|_{\Sc_{q}}^{1-\theta}  \|A \|_{\Sc_{p}}^\theta \leq  \|A \|_{\Sc_{q}}  c_n\big(\Sc^N_{p}\hookrightarrow \Sc^N_{q}\big)^{-\theta}.
	\]
	Rearranging this last bound completes the proof.
\end{proof}

By duality we immediately get the following version for Kolmogorov numbers in the Banach space setting.

\begin{corollary}
\label{cor:Kolmogorov-HiMi-interpolation}
Let $1\leq p, q \leq \infty$, let $\theta\in[0,1]$,
and let $q_\theta$ be the weighted harmonic mean of $p$ and $q$ defined by
\[
	\frac{1}{q_{\theta}} = \frac{1-\theta}{q} + \frac{\theta}{p}.
\]
Assume that $n,N\in \N$ with $1\leq n \leq N^2$.
Then
\[
d_n\big(\Sc^N_{p}\hookrightarrow \Sc^N_{q_\theta}\big) \geq d_n\big(\Sc^N_{p}\hookrightarrow \Sc^N_{q}\big)^{\frac{1/p - 1/q_\theta}{1/p - 1/q}}.
\]
\end{corollary}

Finally, we have the following lemma which describes the behavior of the approximation and Kolmogorov numbers for $0<p\leq 1$, i.e., when the domain space is a quasi-Banach space, showing that those $s$-numbers are stationary at $p=1$.
In the case of $\ell_p$ spaces a similar statement is proved in \cite{MR2432104} (see also \cite[Section 13.1]{MR1393437}).

\begin{lemma}
	\label{lem:approximation-Kolmogorov-p-quasi}
	Let $0 < p \leq 1 \leq q \leq \infty$ and assume that $n,N\in \N$ with $1\leq n \leq N^2$.
	Then
	\begin{equation*}
		a_n\big(\Sc_p^N \hookrightarrow \Sc_q^N\big) = 	a_n\big(\Sc_1^N \hookrightarrow \Sc_q^N\big)
    \end{equation*}
    and
    \begin{equation*}
		d_n\big(\Sc_p^N \hookrightarrow \Sc_q^N\big) = 	d_n\big(\Sc_1^N \hookrightarrow \Sc_q^N\big).	
	\end{equation*}
\end{lemma}

\begin{proof}
	Fix $0 < p \leq 1 \leq q \leq \infty$ and $1\leq n \leq N^2$. 
	We first establish the upper bounds, which are easily obtained by factorization. Indeed, we have
	\begin{align*}
		a_n\big(\Sc_p^N \hookrightarrow \Sc_q^N\big) & \leq \| \Sc_p^N \hookrightarrow \Sc_1^N \| a_n\big(\Sc_1^N \hookrightarrow \Sc_q^N\big) = a_n\big(\Sc_1^N \hookrightarrow \Sc_q^N\big)
	\end{align*}
	and 
	\begin{align*}	
		d_n\big(\Sc_p^N \hookrightarrow \Sc_q^N\big) & \leq \| \Sc_p^N \hookrightarrow \Sc_1^N \|
		d_n\big(\Sc_1^N \hookrightarrow \Sc_q^N\big) = d_n\big(\Sc_1^N \hookrightarrow \Sc_q^N\big).
	\end{align*}

	We continue with the lower bounds and start with the approximation numbers.
	Fix any $\varepsilon>0$ and let $T_n\in \mathscr{L}(\Sc_p^N,\Sc_q^N)$ be such that $\rank T_n <n$ and so that, for all $A\in B_p^N$,
	\begin{equation}
		\label{eq:quasi-conv-appr}
		\| A - T_n A\|_{\Sc_q} \leq (1+\varepsilon) a_n\big(\Sc_p^N \hookrightarrow \Sc_q^N\big).
	\end{equation}
	In order to prove that $a_n(\Sc_1^N \hookrightarrow \Sc_q^N) \leq a_n(\Sc_p^N \hookrightarrow \Sc_q^N)$, we shall show that also, for all $A\in B_1^N $, 
	\begin{equation*}
		 \| A - T_n A\|_{\Sc_q} \leq (1+\varepsilon) a_n\big(\Sc_p^N \hookrightarrow \Sc_q^N\big).
	\end{equation*}
	Then, taking the infimum over all $S\in\mathscr L(\Sc_1^N,\Sc_q^N)$ with $\rank S<n$ and letting $\varepsilon\downarrow 0$, we obtain the desired bound.
	So let us take any linear operator $A\in B_1^N$.
	First, we observe that since $0<p\leq 1$, $B_1^N = \conv(B_p^N)$ (see Lemma \ref{lem:convex hull quasi balls}) and 
	therefore there exists $m\in\N$,
	linear operators $A_1, \ldots, A_m\in B^N_p$,
	and real numbers $\lambda_1,\dots,\lambda_m\geq 0$ with $\sum_{i=1}^m \lambda_i =1$,
	such that $A = \sum_{i=1}^m\lambda_i A_i$.
	Therefore, by \eqref{eq:quasi-conv-appr} applied to $A_1,\dots,A_m$, we obtain
	\[
	\| A - T_n A\|_{\Sc_q}
	\leq \sum_{i=1}^m \lambda_i \| A_i - T_n A_i\|_{\Sc_q}
	\leq \sum_{i=1}^m \lambda_i (1+\varepsilon) a_n\big(\Sc_p^N \hookrightarrow \Sc_q^N\big) 
	= (1+\varepsilon)a_n\big(\Sc_p^N \hookrightarrow \Sc_q^N\big).
	\]
	This shows that indeed $a_n(\Sc_1^N \hookrightarrow \Sc_q^N) = a_n(\Sc_p^N \hookrightarrow \Sc_q^N)$.
	
	The proof for Kolmogorov numbers is almost verbatim the same, but we include it for the sake of completeness and convenience of the reader.
	Fix any $\varepsilon>0$ and let $F\subset \Sc_q^N$ be a subspace with $\dim F <n$ and such that, for all $A\in B_p^N$,
	\begin{equation}
		\label{eq:quasi-conv-kolm}
		\inf_{B\in F} \| A - B\|_{\Sc_q} \leq (1+\varepsilon) d_n\big(\Sc_p^N \hookrightarrow \Sc_q^N\big).
	\end{equation}
	In order to prove that $d_n(\Sc_1^N \hookrightarrow \Sc_q^N) \leq d_n(\Sc_p^N \hookrightarrow \Sc_q^N)$, we shall show that, for all $A\in B_1^N$,
	\begin{equation*}
		 \inf_{B\in F} \| A - B\|_{\Sc_q} \leq (1+\varepsilon) d_n\big(\Sc_p^N \hookrightarrow \Sc_q^N\big).
	\end{equation*}
	To this end, take any linear operator $A\in B_1^N$.
	Again, since $B_1^N = \conv(B_p^N)$ by Lemma \ref{lem:convex hull quasi balls},
	there exists $m\in\N$,
	linear operators $A_1, \ldots, A_m\in B^N_p$,
	and real numbers $\lambda_1,\dots,\lambda_m\geq 0$ with $\sum_{i=1}^m \lambda_i =1$,
	such that $A = \sum_{i=1}^m\lambda_i A_i$.
	For each of the operators $A_1,\dots,A_m$, we use~\eqref{eq:quasi-conv-kolm}  to find $B_i\in F$, such that 
	\begin{equation}
		\label{eq:quasi-conv-kolm-2}
		\| A_i - B_i\|_{\Sc_q} \leq (1+\varepsilon)^2 d_n\big(\Sc_p^N \hookrightarrow \Sc_q^N\big).
	\end{equation}
	Let $B := \sum_{i=1}^m \lambda_i B_i\in F$.
	Then, by \eqref{eq:quasi-conv-kolm-2},
	\[
	\| A - B\|_{\Sc_q}
	\leq \sum_{i=1}^m \lambda_i \| A_i - B_i\|_{\Sc_q}
	\leq \sum_{i=1}^m \lambda_i (1+\varepsilon)^2 d_n\big(\Sc_p^N \hookrightarrow \Sc_q^N\big) 
	= (1+\varepsilon)^2 d_n\big(\Sc_p^N \hookrightarrow \Sc_q^N\big).
	\]
	Taking the infimum over all possible linear operators in $B\in F$ completes the proof that $d_n(\Sc_1^N \hookrightarrow \Sc_q^N) = d_n(\Sc_p^N \hookrightarrow \Sc_q^N)$.
\end{proof}

\begin{remark}
	In the above proof we used the triangle inequality in $\Sc_q^N$.
	Lemma~\ref{lem:approximation-Kolmogorov-p-quasi} is not true in general for $0<p<1$ and $q<1$. Indeed, we shall see in Corollary~\ref{cor:approximation-upper-triangle} below that for $0< q\leq p \leq1$,
	\[
	a_n\big(\Sc_p^N \hookrightarrow \Sc_q^N\big) \asymp_{p,q} \max\Big\{ 1, \frac{N^2-n+1}{N}\Big\}^{1/q-1/p},
	\]
	while 
	$a_n(\Sc_p^N \hookrightarrow \Sc_q^N) =1 $ for $0<p\leq q\leq1$ (see Proposition~\ref{prop:approximation-Kolmogorov-1}).
\end{remark}

\section{Gelfand numbers in the case $0 < p \leq 2 \leq q \leq \infty$ and $0<p < 1$, $p\leq q\leq 2$}
\label{sec:Gelfand}

In this section we improve and complement the results of Hinrichs, the first named author, and  Vyb\'iral \cite{HPV2020}.
First we show that for $0< p \leq 2 \leq q \leq \infty$ the upper bounds provided by the authors in \cite{HPV2020} in the intermediate range $(1-\cHPV) N^2\leq n \leq N^2 - \cHPV N^{1+2/q} +1$ are in fact asymptotically sharp. We shall provide two proofs of which the first is valid for $0<p\leq 2$ with some absolute constant in the lower bound, while the second holds for $1\leq p \leq 2$ but gives $1$ as absolute constant. More precisely, in the first proof we shall follow the approach of Gordon--K\"onig--Sch\"utt.
In order to estimate $c_n\big(\Sc_p^N \hookrightarrow \Sc_{p^*}^N\big)$ for $1\leq p \leq 2$ they used the probabilistic estimate of Lemma~\ref{lem:GKS-FLM}  coupled with the results of Figiel--Lindenstrauss--Milman \cite{FLM1977} on optimal Dvoretzky-dimension of Schatten $p$-classes. In the second proof, we use a relation between $2$-summing norms of invertible linear operators between finite dimensional Banach spaces and their Gelfand numbers.

\begin{proposition}
	\label{prop:Gelfand-missing-intermediate-regime}
	Let $0 < p \leq 2 \leq q \leq \infty$ and assume that $n,N\in \N$ with $1\leq n \leq N^2$.
	Then
	\begin{equation*}
		c_n\big(\Sc_p^N \hookrightarrow \Sc_q^N\big) 
		\gtrsim N^{-1/2-1/p} \sqrt{N^2-n+1}
	\end{equation*}
	whenever $  (1-\cHPV) N^2 \leq n\leq N^2 - \cHPV N^{1+2/q} +1$, where $\cHPV\in(0,1)$ is the universal constant from \cite[Lemma 2.5]{HPV2020}.
\end{proposition}

\begin{proof}
	We will apply Lemma~\ref{lem:GKS-FLM} to $X=\Sc^N_p$, $Y = \Sc^N_q$, $m=N^2$.
	
	Let $(1-\cHPV) N^2 \leq n\leq N^2 - \cHPV N^{1+2/q}$.
	Take any subspace $E\subset \Sc^N_q$ such that $\dim E = N^2-n+1 \geq \cHPV N^{1+2/q}$.
	Note that if $F\subset E$, $\dim F = k$,  and $d(F, \ell^k_2)$, then $k\lesssim N^{1+2/q}$ by Lemma~\ref{lem:Schatten-largest-ell-2-iso}. 
	Hence, by \cite[Inequality (2.10)]{FLM1977},
	\[
	d(E, \ell^{N^2-n+1}_2) \gtrsim \sqrt{\frac{N^2-n+1 }{N^{1+2/q}}} = N^{-1/2-1/q}  \sqrt{N^2-n+1 }.
	\]
	
	Let $A\colon \Sc^N_q \to \Sc^N_2$, $B\colon\Sc^N_2\to \Sc^N_p$ be  the natural embeddings.
	Then
	\[
	\|A \| \| B\| = N^{1/2-1/q} N^{1/p-1/2} = N^{1/p-1/q}.
	\]
	Hence, Lemma~\ref{lem:GKS-FLM} yields
	\[
	c_n\big(\Sc^N_p \hookrightarrow \Sc^N_q\big) \gtrsim N^{-1/2-1/p}  \sqrt{N^2-n+1}.
	\qedhere
	\]
\end{proof}

As mentioned before, when we remain in the setting of Banach spaces, i.e., if $1\leq p \leq 2 \leq q \leq \infty$, then instead of following the approach of Gordon, K\"onig, and Sch\"utt one can use lower bounds on Gelfand numbers involving the $2$--summing norm of $\Sc_q^N\hookrightarrow \Sc_p^N$; the latter can be computed explicitly. In particular, this way we obtain, instead of some absolute constant, the constant $1$ in the lower bound.
\begin{proof}[Alternative proof of Proposition \ref{prop:Gelfand-missing-intermediate-regime} for $1\leq p \leq 2\leq q \leq \infty$]
Recall that the $2$-summing norm (which is always greater than or equal to the operator norm) of a linear operator $T$ between two Banach spaces $X$ and $Y$ is defined as
\[
\pi_2(T) := \sup\Bigg\{ \Big(\sum_{i=1}^n\|Tx_i\|_Y^2\Big)^{1/2}\,:\, n\in\N,\,x_1,\dots,x_n\in X,\, \sup_{\|x^*\|=1}\Big(\sum_{i=1}^n|x^*(x_i)|^2\Big)^{1/2} \leq 1 \Bigg\} \in[0,\infty].
\] 
One can show that (see, e.g., \cite[Lemma 5.2]{DMM2006})
\[
\pi_2\big(\Sc_p^N\hookrightarrow\Sc_q^N\big) = N \,\frac{\max\big\{1,N^{1/q-1/2} \big\}}{\max\big\{1,N^{1/p-1/2}\big \}}.
\]
Combining this with the fact (see \cite[Lemma, p. 231]{CD1997}) that for any invertible linear operator $T$ between $m$-dimensional Banach spaces $X$ and $Y$ and any $1\leq n \leq m$ it holds that
\[
c_n\big( T:X\to Y\big) \geq \frac{\sqrt{m-n+1}}{\pi_2(T^{-1})},
\]
we obtain the lower bound of Proposition~\ref{prop:Gelfand-missing-intermediate-regime} with constant $1$.
\end{proof}

In \cite{HPV2020} the lower bound for small codimensions was based on estimates of Kolmogorov numbers of balls in mixed norm spaces due to Vasil'eva \cite{V2013}. Below we present an alternative proof (based again on the results of Gordon--K\"onig--Sch\"utt \cite{GKS1987})
which gives absolute constants rather than constants that depend on $p$ and $q$ as in \cite[Proposition 4.9]{HPV2020}.

\begin{proposition}
        \label{prop:Gelfand-cleaner-small-regime}
        Let $1\leq p \leq 2 \leq q \leq \infty$ and assume that $n,N\in \N$ with $1\leq n \leq N^2$.
        Then
        \begin{equation*}
                c_n\big(\Sc_p^N \hookrightarrow \Sc_q^N\big)
                \gtrsim \min\Bigl\{1, \frac{N^{3/2-1/p}}{n^{1/2}} \Bigr\}
        \end{equation*}
        for  $ 1\leq n\leq (1-\cHPV) N^2$, where $\cHPV\in(0,1)$ is the universal constant from \cite[Lemma 2.5]{HPV2020}.
\end{proposition}

\begin{proof}
By duality,
\[
c_n\big(\Sc^N_p\hookrightarrow \Sc^N_q\big) = d_n\big(\Sc^N_{q^*}\hookrightarrow \Sc^N_{p^*}\big).
\]
We now apply Lemma~\ref{lem:GKS-Kolmogorov-lower-bound} to $X=\Sc^N_{q^*}$, $Y = \Sc^N_{p^*}$, $m=N^2$, and with $I$ being the natural embedding of $Y = \Sc^N_{p^*}$ into $\Sc^N_2$.
Since $1\leq n\leq (1-\cHPV )N^2$ and $\|\Sc^N_{p^*}\hookrightarrow \Sc^N_2\| = N^{1/2-1/p^*} = N^{-1/2+1/p}$, we obtain
\[
 d_n\big(\Sc^N_{q^*}\hookrightarrow \Sc^N_{p^*}\big) \geq
        \frac{N^2 - \sqrt{(n-1)N^2}}{\sqrt{(n-1)N^2} N^{-1/2+1/p} + N^2}
         \gtrsim
          \frac{N^2}{\sqrt{n} N^{1/2+1/p} + N^2} \gtrsim \min\Bigl\{ 1, \frac{N^{3/2-1/p}}{n^{1/2}} \Bigr\}.
        \]
This completes the proof.
\end{proof}

Finally, we give an elementary proof of the following result of Ch\'avez-Dom\'inguez and Kutzarova~\cite{CDK2015}.

\begin{proposition}[{\cite[Theorem 5.1]{CDK2015}}]
\label{prop:Gelfand-CDK}
Let $0 < p \leq 1$, $p \leq q \leq 2$ and assume that $n,N\in \N$ with $1\leq n \leq N^2$.
Then
\begin{equation*}
c_n\big(\Sc_p^N \hookrightarrow \Sc_q^N\big) 
\asymp_{p,q}
\min\Bigl\{1,\Bigl(\frac{N}{n}\Bigr)^{1/p-1/q} \Bigr\}.
 \end{equation*}
Moreover, the lower bound carries over to $q>2$.
\end{proposition}

\begin{proof}
We start with the upper bound.
We know that $ c_n(\Sc^N_{1}\hookrightarrow \Sc^N_{2})  \asymp \min\bigl\{ 1, (\frac{N}{n})^{1/2} \bigr\}$, see, e.g., \cite{CD1997}.
Thus, by Lemma~\ref{lem:Gelfand-HiMi-interpolation} (used with $q=2$, $p\in(0,1)$, and $p_\theta=1$),
\[
 c_n\big(\Sc^N_{p}\hookrightarrow \Sc^N_{2}\big) \leq c_n\big(\Sc^N_{1}\hookrightarrow \Sc^N_{2}\big)^{\frac{1/p-1/2}{1-1/2}} \lesssim_p  \min\Bigl\{ 1, \Bigl(\frac{N}{n}\Bigr)^{1/p-1/2} \Bigr\}
\]
for $0<p<1$.
Now we are in position to interpolate in the codomain space (similarly as Hinrichs, the first named author, and Vyb\'iral \cite{HPV2020} in the case $1\leq p \leq q \leq 2$):
 by \eqref{eq:Gelfand-interpolation} (used with $p\in (0,1)$, $q=2$, $q_0=q\in [p,2] $),
\[
 c_n\big(\Sc^N_{p}\hookrightarrow \Sc^N_{q}\big) \leq c_n\big(\Sc^N_{p}\hookrightarrow \Sc^N_{2}\big)^{\frac{1/p-1/q}{1/p-1/2}} \lesssim_{p,q}  \min\Bigl\{ 1, \Bigl(\frac{N}{n}\Bigr)^{1/p-1/q} \Bigr\}
\]
for $0<p<1$, $p\leq q \leq 2$.

As for the lower bound, let $0<p\leq 1$, $p\leq q \leq \infty$ and assume that $n,N\in \N$ with $1\leq n \leq N^2$.
Then
\[
1 = c_{N^2}\big(\Sc_p^N\hookrightarrow \Sc_p^N\big)
 \leq 
 c_n\big(\Sc_p^N\hookrightarrow \Sc_q^N\big) 
 c_{N^2-n+1}\big(\Sc_q^N\hookrightarrow \Sc_p^N\big).
\]
Hence,
\[		
c_n\big(\Sc_p^N\hookrightarrow \Sc_q^N\big)
		 \geq \frac{1}{c_{N^2-n+1}\big(\Sc_q^N\hookrightarrow \Sc_p^N\big)}  
		\asymp_{p,q} \frac{1}{ \max\bigl\{ 1, \frac{n}{N}\bigr\}^{1/p-1/q} }= \min\Bigl\{ 1, \frac{N}{n}\Bigr\}^{1/p-1/q},
\]
	where we used the results of \cite{HPV2020}.
%
\end{proof}

\section{Approximation and Kolmogorov numbers in the case $0<q\leq p \leq \infty$}
\label{sec:approximation-upper-triangle}

The case $0<q\leq p \leq \infty$ is relatively easy and we remark that in the literature there are results and conjectures about the asymptotics formulas for $a_n(\id\colon E_m \to F_m)$ and $c_n(\id\colon E_m \to F_m)$, where $E$ and $F$ are symmetric sequence spaces and $F$ is naturally embedded into $E$, see, e.g., \cite{MR1861746,MR1932582,HM2005,MR2223589}. 

Recall that in the case of $\ell_p$ spaces the exact values of $a_n(\ell_p^N\hookrightarrow \ell_q^N)$ and  $c_n(\ell_p^N\hookrightarrow \ell_q^N)$ are found by considering just the projections onto or restrictions to coordinate subspaces.
This suggests using subspaces containing only matrices with few nonzero singular values to estimate $s$-numbers of the embeddings $\Sc_p^N \hookrightarrow \Sc_q^N$.
In the case of Gelfand numbers in the range $0<q\leq p \leq \infty$ this idea was executed in \cite{HPV2020}  (extending the results of \cite[Proposition 4.1, Example 4.7]{HM2005} concerning  $q \leq p \leq \infty$ with $1\leq q \leq 2$).
Below we slightly adapt the proof of the upper bound  for the Gelfand numbers to get the same upper bound approximation numbers; again the case  $1\leq q\leq p \leq \infty$ appears already in \cite[Proposition 4.1, Corollary 4.8]{HM2005}.

\begin{proposition}
	\label{prop:approximation-upper-triangle-upper-bd}
	Let $0 < q \leq p\leq \infty$ and assume that $n,N\in \N$ with $1\leq n \leq N^2$.
	Then
	\begin{equation*}
		a_n\big(\Sc_p^N \hookrightarrow \Sc_q^N\big)  \lesssim_{p,q} \max\Bigl\{ 1, \frac{N^2-n+1}{N}\Bigr\}^{1/q-1/p}.
	\end{equation*}
\end{proposition}

By the results of \cite{HPV2020} this immediately yields the following corollary.

\begin{corollary}
	\label{cor:approximation-upper-triangle}
	Let $0 < q \leq p\leq \infty$ and assume that $n,N\in \N$ with $1\leq n \leq N^2$.
	Then
	\begin{equation*}
		a_n\big(\Sc_p^N \hookrightarrow \Sc_q^N\big)  \asymp_{p,q} c_n\big(\Sc_p^N \hookrightarrow \Sc_q^N\big)  \asymp_{p,q}  \max\Bigl\{ 1, \frac{N^2-n+1}{N}\Bigr\}^{1/q-1/p}.
	\end{equation*}
\end{corollary}

\begin{proof}[Proof of  Proposition~\ref{prop:approximation-upper-triangle-upper-bd}]
	We follow the proof of  \cite[Proposition~3.1]{HPV2020}.
	Let $M$ be the linear subspace of all $N\times N$ real matrices with all entries in the first $k$ columns equal to $0$.
	Since all matrices in $M$ have rank at most $N-k$, the number of their nonzero singular values is at most $N-k$.
	Thus H\"older's inequality implies that 
	\begin{equation}
		\label{eq:Holder-zero-rows}
		\|B\|_{\Sc_q} \leq (N-k)^{1/q-1/p} \|B\|_{\Sc_p}  
	\end{equation}
	for any matrix $B\in M$.
	
	Let $P\colon \Sc^N_2\to  \Sc^N_2$ be the orthogonal projection on $M^\perp$.
	Then $\rank P = \codim M = k N$ and $\id - P$ is the orthogonal projection onto $M$.
	For any $A\in \Sc^N_p$ we have
	\[
	\| (\id - P) A \|_{\Sc_q} \leq (N-k)^{1/q-1/p}  \| (\id - P) A \|_{\Sc_p} \leq (N-k)^{1/q-1/p}  \| A \|_{\Sc_p},
	\]
	where we first used inequality~\eqref{eq:Holder-zero-rows} and then the fact that deleting columns (or, equivalently, replacing them with zeros) decreases the singular values, see \cite[Corollary~7.3.6]{MR2978290}.
	Hence,
	\[
	a_n\big(\Sc_p^N\hookrightarrow\Sc_q^N\big)
	\leq (N-k)^{1/q-1/p}
	\]
	whenever $k N < n$. 
	Choosing $k=\left\lfloor \frac{n-1}{N} \right\rfloor$, we have $N-k = \left\lceil \frac{N^2-n+1}{N} \right\rceil$ and obtain
	\[
	a_n\big(\Sc_p^N\hookrightarrow \Sc_q^N\big)
	\leq \left\lceil \frac{N^2-n+1}{N} \right\rceil^{1/q-1/p}.
	\]
	This easily translates into the claimed estimate.
\end{proof}

Let us now comment on Kolmogorov numbers in the case $0 < q \leq p\leq \infty$.
Clearly, if  $1\leq q \leq p \leq \infty$, then $1\leq p^* \leq q^* \leq \infty$ and by duality
\[
d_n\big(\Sc_p^N \hookrightarrow \Sc_q^N\big)  = c_n\big(\Sc_{q^*}^N \hookrightarrow \Sc_{p^*}^N\big) \asymp_{p,q} \max\Bigl\{ 1, \frac{N^2-n+1}{N}\Bigr\}^{1/q-1/p}.
\]
The case $0<q<1$, $q\leq p$ is somewhat more delicate, even for embeddings  $\ell_p^N\hookrightarrow \ell_q^N$. 
Apart from the upper bound
\[
d_n\big(\Sc_p^N \hookrightarrow \Sc_q^N\big)  \leq a_n\big(\Sc_{p}^N \hookrightarrow \Sc_{q}^N\big) \asymp_{p,q} \max\Bigl\{ 1, \frac{N^2-n+1}{N}\Bigr\}^{1/q-1/p},
\]
we have the following analogue of the results obtained by Vyb\'iral in \cite[Section 4]{MR2432104}.

\begin{proposition}
	\label{prop:kolmogorov-upper-triangle}
	Let $0 < q \leq p\leq \infty$ and assume that $n,N\in \N$ with $1\leq n \leq N^2$.
    If $1\leq q \leq p \leq \infty$, then
	\begin{equation*}
		d_n\big(\Sc_p^N \hookrightarrow \Sc_q^N\big) 
		\asymp_{p,q}
			\max\Bigl\{ 1, \frac{N^2-n+1}{N}\Bigr\}^{1/q-1/p}
	\end{equation*}
	(and the upper bound carries over to all $0< q \leq p \leq \infty$).
Moreover, if $0<q<1$ and $q\leq p$, then there exists a constant $c(p,q)\in(0,2]$ such that
	\begin{equation*}
	d_{\lfloor c(p,q)N^2\rfloor+1}\big(\Sc_p^{2N} \hookrightarrow \Sc_q^{2N}\big) 
	\gtrsim_{p,q} N^{1/q-1/p}.
\end{equation*}
\end{proposition}

\begin{remark}
If we knew that for some $0< q_0 < p_0 \leq \infty$ such that $0<q_0<1$ the lower bound 
\[
d_n\big(\Sc_{p_0}^N \hookrightarrow \Sc_{q_0}^N\big) 
\gtrsim_{p_0,q_0}
\max\Bigl\{ 1, \frac{N^2-n+1}{N}\Bigr\}^{1/q_0-1/p_0}
\]
holds, 
then the lower bound 
 \[
 d_n\big(\Sc_{p}^N \hookrightarrow \Sc_{q_0}^N\big) 
 \gtrsim_{p,p_0, q_0}
 \max\Bigl\{ 1, \frac{N^2-n+1}{N}\Bigr\}^{1/q_0-1/p}
 \]
 for all  $ p \geq p_0$ would follow from the  interpolation inequality \eqref{eq:Kolmogorov-interpolation}.
 \end{remark}

\begin{proof}[Proof of  Proposition~\ref{prop:kolmogorov-upper-triangle}]
	As mentioned above, if $1\leq q \leq p \leq \infty$, then the result follows by duality (and the upper bound is alway true by comparison with the approximation numbers) . Let $0<q<1$ and $q\leq p$. By Carl's inequality for quasi-Banach spaces \cite{MR1339817},
	\[
	\sup_{k\leq 2N^2} k^\alpha e_k\big(\Sc_p^{2N} \hookrightarrow \Sc_q^{2N}\big)
	 \lesssim_{\alpha, p, q} 
	 \sup_{k\leq 2N^2} k^\alpha d_k\big(\Sc_p^{2N} \hookrightarrow \Sc_q^{2N}\big). 
	\]
    By the estimates of entropy numbers of embeddings of Schatten classes from \cite{HPV2017}, the left-hand side can be bounded as follows,
    \[
\sup_{k\leq 2N^2} k^\alpha e_k\big(\Sc_p^{2N} \hookrightarrow \Sc_q^{2N}\big)
\geq (2N^2)^{\alpha} e_{2N^2}\big(\Sc_p^{2N} \hookrightarrow \Sc_q^{2N}\big) 
\gtrsim_{\alpha,p,q} N^{2\alpha} N^{1/q-1/p}.
\]
As for the right-hand side, there exists $1\leq k_N \leq 2N^2$ such that
\[
	 \sup_{k\leq 2N^2} k^\alpha d_k\big(\Sc_p^{2N} \hookrightarrow \Sc_q^{2N}\big) 
	 = k_N^\alpha d_{k_N}\big(\Sc_p^{2N} \hookrightarrow \Sc_q^{2N}\big)
	  \leq k_N^\alpha (2N)^{1/q-1/p}.
\]
Together this yields $N^{2\alpha} N^{1/q-1/p} \lesssim_{\alpha,p,q} k_N^\alpha d_{k_N}\big(\Sc_p^{2N} \hookrightarrow \Sc_q^{2N}\big)\lesssim_{\alpha,p,q} k_N^\alpha N^{1/q-1/p}$.
Since this holds for any $N\in\N$, we conclude that there exists a constant $c(p,q)\in (0,2]$ such that $ c(p,q)N^2 \leq k_N \leq 2N^2$.
This implies the assertion.
\end{proof}

\section{Approximation numbers in the case $1\leq p\leq 2 \leq q  \leq \infty$} 
\label{sec:approximation-square}

\subsection{Notation}

For $1  \leq p \leq 2 \leq q \leq \infty$ let us introduce the two quantities
\begin{equation*}
\critmax_{p,q} := \max\bigl\{ 3-2/p,1+2/q\bigr\}\qquad\text{and}\qquad
\critmin_{p,q}:= \min\bigl\{ 3-2/p,1+2/q\bigr\}.
\end{equation*}
Note that 
in this case 
we have $1\leq \critmin_{p,q} \leq \critmax_{p,q} \leq 2$ and that $\critmax_{p,q}=\critmin_{p,q} = 3-2/p$ whenever $1/p+1/q=1$.
In what follows, we shall frequently use the relations $1+2/p^* = 3-2/p$ and $1/q-1/p = 1/p^* - 1/q^*$.

Our conjecture  is that for $1\leq p\leq 2 \leq q  \leq \infty$ we have
\begin{align*}
a_n\big(\Sc_p^N \hookrightarrow \Sc_q^N\big) 
&\asymp_{p,q}
\begin{cases}
\min\Bigl\{ 1, \frac{N^{\critmax_{p,q}/2}}{n^{1/2}} \Bigr\}
          &:\  1 \leq n \leq (1-\cHPV ) N^2,\\
N^{\critmax_{p,q}/2-2} \sqrt{N^2-n}
          &:\ (1-\cHPV ) N^2 \leq n\leq N^2 -  \cHPV  N^{\critmin_{p,q}}+1,\\
N^{1/q-1/p} 
          &:\ N^2 -\cHPV  N^{\critmin_{p,q}} +1\leq n \leq N^2
\end{cases}
\end{align*}
(recall that $\cHPV \in (0,1)$ stands for the universal constant from \cite[Lemma 2.5]{HPV2020}).
We stress that the essential case is the one where $1/p+1/q =1$, since the upper bound in other cases can be obtained by factorization and comparing the value of $a_n(\Sc_p^N \hookrightarrow \Sc_{q}^N) $ with $a_n(\Sc_r^N \hookrightarrow \Sc_{r^*}^N)$ for some appropriately chosen $r$. 

The main result of this section is Proposition~\ref{prop:approximation-square-upper-bd-large-n}, which removes the logarithms from the estimate of  Gordon, K\"onig, and Sch\"utt  \cite{GKS1987} in the case when $n$ is large.

\subsection{Lower bound}

\begin{proposition}
\label{prop:approximation-square-lower-bd}
Let $1\leq p\leq 2\leq q \leq \infty$ and assume that $n,N\in \N$ with $1\leq n \leq N^2$.
Then
\begin{align*}
a_n\big(\Sc_p^N \hookrightarrow \Sc_q^N\big) 
&\gtrsim
\begin{cases}
\min\Bigl\{ 1, \frac{N^{\critmax_{p,q}/2}}{n^{1/2}} \Bigr\}
          &:\  1 \leq n \leq (1-\cHPV) N^2,\\
N^{\critmax_{p,q}/2-2} \sqrt{N^2-n}
          &:\ (1-\cHPV) N^2 \leq n\leq N^2 -  \cHPV N^{\critmin_{p,q}}+1,\\
N^{1/q-1/p} 
          &:\ N^2 - \cHPV N^{\critmin_{p,q}} +1\leq n \leq N^2,
\end{cases}
\end{align*}
where $\cHPV \in (0,1)$ is the universal constant from \cite[Lemma 2.5]{HPV2020}. 
The implicit constant in the estimate is universal and does not depend on $p$ and $q$.
\end{proposition}

\begin{proof}
We have
\[
a_n\big(\Sc_p^N \hookrightarrow \Sc_q^N\big) 
\geq \max\bigl\{ c_n\big(\Sc_p^N \hookrightarrow \Sc_q^N\big) , d_n\big(\Sc_p^N \hookrightarrow \Sc_q^N\big) 
 \bigr\}
= \max\bigl\{ c_n\big(\Sc_p^N \hookrightarrow \Sc_q^N\big) , c_n\big(\Sc_{q^*}^N \hookrightarrow \Sc_{p^*}^N\big) 
 \bigr\}.
 \]
  The result follows immediately from Propositions~\ref{prop:Gelfand-missing-intermediate-regime}, \ref{prop:Gelfand-cleaner-small-regime}, the definitons of $\critmax_{p,q}$, $\critmin_{p,q}$, and the trivial estimate
  \[
  c_n\big(\Sc_p^N \hookrightarrow \Sc_q^N\big) 
  \geq c_{N^2}\big(\Sc_p^N \hookrightarrow \Sc_q^N\big) 
   = \frac{1}{\|\Sc_q^N \hookrightarrow \Sc_p^N\| } = N^{1/q-1/p}.
  \qedhere
  \]
  \end{proof}

\subsection{Upper bound for small $n$}

\begin{proposition}
\label{prop:approximation-square-upper-bd-small-n}
Let $1\leq p\leq 2\leq q \leq \infty$ and assume that $n,N\in \N$ with $1\leq n \leq N^2$.
If $1 \leq n\leq (1-\cHPV)N^2$, 
where $\cHPV \in (0,1)$ is the universal constant from \cite[Lemma 2.5]{HPV2020},
then
\begin{equation*}
a_n\big(\Sc_p^N \hookrightarrow \Sc_{q}^N\big) 
\lesssim_{p,q}
\frac{N^{\critmax_{p,q}/2}}{n^{1/2}}.
\end{equation*}
\end{proposition}

\begin{proof}
By duality, without loss of generality, we may and do assume that $1\leq p\leq 2 \leq q \leq \infty$ satisfy $1/p+1/q\leq 1$ (equivalently, $1/p \leq 1/q^*$, or $1+2/q \leq 3-2/p$, or $\critmax_{p,q}=3-2/p$).
By factorization and \cite[Proposition~3.7]{GKS1987} (or \cite[Proposition~3.8]{GKS1987} if $q=\infty$; note that instead of the range $1\leq n\leq N^2/2$ given in the formulations of those results we prefer to use the range $1\leq n\leq (1-c)N^2$, but the statements obviously remain valid),
\begin{align*}
a_n\big(\Sc_p^N \hookrightarrow \Sc_q^N\big) 
& \leq \| \Sc_p^N \hookrightarrow \Sc_{q^*}^N  \| a_n\big(\Sc_{q^*}^N \hookrightarrow \Sc_q^N\big)  \\
&\lesssim_q N^{1/q^* - 1/p} \min\Bigl\{ 1, \frac{N^{3/2-1/q^*}}{n^{1/2}} \Bigr\} \\
& =  \min\Bigl\{ N^{1/q^* - 1/p}, \frac{N^{3/2-1/p}}{n^{1/2}} \Bigr\}.
\end{align*}
Moreover, we have the trivial upper bound $a_n\big(\Sc_p^N \hookrightarrow \Sc_q^N\big)\leq \| \Sc_p^N \hookrightarrow \Sc_{q}^N  \|=1 $.
This yields the assertion.
\end{proof}

\subsection{Upper bound for large $n$}

The next proposition shows how one can remove the logarithm from the estimate of  Gordon, K\"onig, and Sch\"utt  \cite[Proposition~3.7]{GKS1987} in the case when $n$ is large.

\begin{proposition}
\label{prop:approximation-square-upper-bd-large-n}
Let $1\leq p\leq 2$ and assume that $n,N\in \N$ with $1\leq n \leq N^2$.
If $N^2 - c N^{3-2/p} +1\leq n \leq N^2$,
where $\cHPV \in (0,1)$ is the universal constant from \cite[Lemma 2.5]{HPV2020},
 then
\begin{align*}
a_n\big(\Sc_p^N \hookrightarrow \Sc_{p^*}^N\big) 
\lesssim_{p}
N^{1-2/p}.
\end{align*}
Consequently, if $1\leq p \leq 2 \leq q \leq \infty$, then
\begin{align*}
a_n\big(\Sc_p^N \hookrightarrow \Sc_{q}^N\big) 
\lesssim_{p,q}
\begin{cases}
N^{\critmax_{p,q}/2-2} \sqrt{N^2-n+1}
          &:\, N^2 - \cHPV  N^{\critmax_{p,q}} +1 \leq n\leq N^2 -  \cHPV N^{\critmin_{p,q}}+1,\\
N^{1/q-1/p} 
          &:\,  N^2 - \cHPV  N^{\critmin_{p,q}} +1\leq n \leq N^2.\end{cases}
\end{align*}
\end{proposition}

\begin{proof}
 The proof will be divided into two steps.

\emph{Step 1.}
First, we consider the special case $q=p^*$. 
Let $N^2 - \cHPV  N^{3-2/p} +1\leq n \leq N^2$.
From the proof of \cite[Proposition 4.1]{HPV2020} 
(in the case of the embedding $\Sc_2^N \hookrightarrow \Sc_{p^*}^N$)
we know that there exists a subspace $L\subset \Sc_{p^*}^N$ with $\codim L <n$, such that
\[
c_n \big(\Sc_2^N \hookrightarrow \Sc_{p^*}^N\big) 
\leq \| \id |_L \colon L\subset \Sc_2^N \to \Sc_{p^*}^N\| \leq c_1(p^*) c_2 N^{1/p^* -1/2},
\]
where the function $r\mapsto c_1(r)\in(0,\infty)$ is monotone increasing in $r\in [2,\infty]$.
We also have
\[
 N^{1/p^* -1/2} \leq c_n \big(\Sc_2^N \hookrightarrow \Sc_{p^*}^N\big) = a_n \big(\Sc_2^N \hookrightarrow \Sc_{p^*}^N\big)
\]
(the Gelfand and approximation numbers are equal since the domain is a Hilbert space; the inequality follows from the trivial bound by $c_{N^2}(\Sc_2^N \hookrightarrow \Sc_{p^*}^N) $) and so
\begin{equation}
\label{eq:appr-2-p-prim-large}
 N^{1/p^* -1/2}
 \leq a_n \big(\Sc_2^N \hookrightarrow \Sc_{p^*}^N\big)
 \leq \| \id |_L \colon L\subset \Sc_2^N \to \Sc_{p^*}^N\| 
\lesssim_p N^{1/p^* -1/2},
\end{equation}
where the constant in the last inequality increases if $p^*$ increases.

Denote by $Q\colon \Sc_2^N \to \Sc_2^N$ the orthogonal projection onto $L^\perp \subset \Sc_2^N$ 
(we can algebraically treat $L$ as a subspace of $\Sc_2^N$;
similarly, we will consider $Q$ as a linear operator between different Schatten classes). 
Then $\rank(Q) < n$, $P = \id - Q$ is the orthogonal projection onto $L$, and  $P = P^2 = P^*$. 
Moreover,
\begin{equation}
\label{eq:gal}
\| \id - Q \colon \Sc_2^N \to \Sc_{p^*}^N\| = \| \id |_L \colon L\subset \Sc_2^N \to \Sc_{p^*}^N\|.
\end{equation}
Indeed, the inequality  $\| \id - Q \colon \Sc_2^N \to \Sc_{p^*}^N\| 
\geq  \| \id |_L \colon L\subset \Sc_2^N \to \Sc_{p^*}^N\|$ is obvious, since $\id - Q = \id|_L$ on $L$; 
on the other hand, since $ \|(\id - Q)A\|_{\Sc_{2}}\leq  \|A\|_{\Sc_{2}}$ for $A\in  \Sc_{2}^N$,
\begin{align*}
\| \id - Q \colon \Sc_2^N \to \Sc_{p^*}^N\| 
&= \sup \{ \| (\id - Q) A\|_{\Sc_{p^*}}  : \|A\|_{\Sc_{2}}\leq 1 \}   \\
&\leq \sup \{ \| (\id - Q) A\|_{\Sc_{p^*}}  : \|(\id - Q)A\|_{\Sc_{2}}\leq 1 \}   \\
&\leq \sup \{ \|B\|_{\Sc_{p^*}}  : \|B\|_{\Sc_{2}}\leq 1, \quad B\in L \}  \\
&=  \| \id |_L \colon L\subset \Sc_2^N \to \Sc_{p^*}^N\|.
\end{align*}

Thus (recall that $\rank(Q) < n$ and $\id - Q = P = P^* = P^2$),
\begin{align*}
a_n \big(\Sc_p^N \hookrightarrow \Sc_{p^*}^N\big)
 \leq \|  \id - Q \colon \Sc_p^N \to \Sc_{p^*}^N\| 
&= \|  P^2  \colon \Sc_p^N \to \Sc_{p^*}^N\| \\
&\leq \|  P  \colon \Sc_p^N \to \Sc_{2}^N\| \cdot \| P  \colon \Sc_2^N \to \Sc_{p^*}^N\| \\
&= \|P  \colon \Sc_2^N \to \Sc_{p^*}^N\|^2 =  \| \id - Q  \colon \Sc_2^N \to \Sc_{p^*}^N\|^2.
\end{align*}
This together with \eqref{eq:appr-2-p-prim-large}, \eqref{eq:gal} yields
\begin{equation*}
a_n \big(\Sc_p^N \hookrightarrow \Sc_{p^*}^N\big) \lesssim_p N^{2/p^* -1} = N^{1-2/p}
\end{equation*}
for $N^2 - \cHPV  N^{3-2/p} +1\leq n \leq N^2$ and ends the proof in Step 1. Note that the constant here increases if $p^*$ increases.

\emph{Step 2.}
Take now any $1\leq p < 2 < q \leq \infty$.
By duality we may and do assume that $1/p+1/q\leq 1$.
Equivalently, $1+2/q \leq 3-2/p$, i.e.,  $\critmax_{p,q}=3-2/p$ and $\critmin_{p,q} = 1+2/q$.
We split the reasoning into two substeps.

\emph{Step 2a.}
Let $N^2- \cHPV N^{1+2/q} +1 \leq n \leq N^2$.
By factorization,
\[
a_n\big(\Sc_p^N \hookrightarrow \Sc_q^N\big)  \leq \| \Sc_p^N \to \Sc_{q^*}^N  \| a_n\big(\Sc_{q^*}^N \hookrightarrow \Sc_q^N\big)  = N^{1/q^* - 1/p} a_n\big(\Sc_{q^*}^N \hookrightarrow \Sc_q^N\big).
\]
By the results of Step 1, $a_n(\Sc_{q^*}^N \hookrightarrow \Sc_q^N) \lesssim_q N^{1/q-1/q^*}$
 for $N^2- \cHPV N^{3 - 2/q^*} +1  \leq n \leq N^2$.
We conclude that for $ N^2- \cHPV N^{1+2/q} +1\leq n \leq N^2$,
\[
a_n\big(\Sc_p^N \hookrightarrow \Sc_q^N\big)  \lesssim_{q}  N^{1/q - 1/p}.
\]
Note that the constant here increases if $q$ increases.

\emph{Step 2b.}
Let now $N^2- \cHPV N^{3-2/p} +1 \leq n \leq N^2 - \cHPV N^{1+2/q}$.
As in the proof of \cite[Proposition~4.1]{HPV2020}, we shall apply a Gluskin-type trick and use the result of Step 2a for $a_n (\Sc_p^N \hookrightarrow \Sc_{s}^N)$ with an appropriately chosen value of $s$.
However, since we want to stay in the regime where $a_n(\Sc_p^N \hookrightarrow \Sc_{s}^N) \asymp c(\Sc_p^N \hookrightarrow \Sc_{s}^N)$ we get an additional constraint for $s$ which does not appear in the case of Gelfand numbers.
Hence, we do not obtain results for the whole intermediate range $N^2- \cHPV N^2 +1 \leq n \leq N^2 - \cHPV N^{1+2/q}$.

Recall that $3-2/p = 1+2/p^*$.
Therefore, we can choose $s\in[p^*,q]$ such that $n = N^2 - \cHPV N^{1+2/s}+1$.
By the results of Step 2a, 
\[
a_n\big(\Sc_p^N \hookrightarrow \Sc_q^N\big)  \leq a_n\big(\Sc_{p}^N \hookrightarrow \Sc_s^N\big)  \lesssim_{p,s} N^{1/s - 1/p} = N^{-1/2 - 1/p} \sqrt{N^2 - n+1}.
\]
Note that the constants depend on $p$ and $s$, but since $s\leq q$ and the constants in Step 1 and 2 were appropriately monotone, the constant here depends in fact only on $p$ and $q$.
\end{proof}

\subsection{Intermediate $n$}

In the case $1\leq p\leq 2\leq q=p^* \leq \infty$, in the intermediate range $(1-\cHPV) N^2 \leq n \leq N^2-\cHPV N^{3-2/p} +1$, there are some  logarithmic factors in the estimate of  Gordon, K\"onig, and Sch\"utt  \cite[Proposition~3.7]{GKS1987}.

\begin{proposition}
	\label{prop:approximation-square-upper-bd-intermediate-n}
	Let $1\leq p\leq 2$ and assume that $n,N\in \N$ with $1\leq n \leq N^2$.
	If $(1-\cHPV )N^2 \leq n\leq N^2 - \cHPV  N^{3-2/p} +1$, 
	where $\cHPV \in (0,1)$ is the universal constant from \cite[Lemma 2.5]{HPV2020},
	then
	\begin{equation*}
		a_n\big(\Sc_p^N \hookrightarrow \Sc_{p^*}^N\big) 
		\lesssim_{p} 
		N^{-1/2-1/p} \sqrt{N^2-n+1} \sqrt{\log N}
	\end{equation*}	
	Consequently, if $1\leq p \leq 2 \leq q \leq \infty$ and $(1-\cHPV )N^2 \leq n\leq N^2 - \cHPV  N^{\critmax_{p,q}} +1$, then
	\begin{equation*}
		a_n\big(\Sc_p^N \hookrightarrow \Sc_{q}^N\big) 
		\lesssim_{p,q} 
		N^{\critmax_{p,q}/2-2} \sqrt{N^2-n+1} \sqrt{\log N}
			\end{equation*}
\end{proposition}

However, in the case $p =1$ and $q=p^* = \infty$  Gordon, K\"onig, and Sch\"utt  \cite[Proposition~3.8]{GKS1987} have an exact estimate.
By factorization, it extends to $p=1$ and $2\leq q\leq\infty$ (or $1\leq p\leq 2$ and $q=\infty$) and yields the following asymptotically sharp upper bounds (the corresponding lower bounds are the content of Proposition \ref{prop:approximation-square-lower-bd}).

\begin{proposition}
\label{prop:approximation-square-upper-bd-special}
Let $1\leq p\leq 2\leq q \leq \infty$ and assume that $n,N\in \N$ with $1\leq n \leq N^2$.
Then
\begin{align*}
a_n\big(\Sc_1^N \hookrightarrow \Sc_q^N\big)
&\lesssim_q
\begin{cases}
\min\Bigl\{ 1, \frac{N^{1/2+1/q}}{n^{1/2}} \Bigr\}
          &:\  1 \leq n \leq (1-\cHPV) N^2,\\
N^{-3/2+1/q} \sqrt{N^2-n+1}
          &:\ (1-\cHPV) N^2 \leq n\leq N^2 -  \cHPV N+1,\\
N^{1/q-1} 
          &:\ N^2 - \cHPV N +1\leq n \leq N^2,
\end{cases}
\\
a_n\big(\Sc_p^N \hookrightarrow \Sc_\infty^N\big)
&\lesssim_p
\begin{cases}
\min\Bigl\{ 1, \frac{N^{3/2-1/p}}{n^{1/2}} \Bigr\}
          &:\  1 \leq n \leq (1-\cHPV) N^2,\\
N^{-1/2-1/p} \sqrt{N^2-n+1}
          &:\ (1-\cHPV) N^2 \leq n\leq N^2 -  \cHPV N+1,\\
N^{-1/p} 
          &:\ N^2 - \cHPV N +1\leq n \leq N^2.
\end{cases}
\end{align*}
Here $\cHPV \in (0,1)$ is the universal constant from \cite[Lemma 2.5]{HPV2020}.
\end{proposition}

\begin{proof}[Proof of Proposition~\ref{prop:approximation-square-upper-bd-intermediate-n}]
	\emph{Step 1.}
	The estimate of $a_n\big(\Sc_p^N \hookrightarrow \Sc_{p^*}^N\big)$ is the content of \cite[Proposition~3.7]{GKS1987};
	 the fact that we consider the range $(1-\cHPV )N^2 \leq n\leq N^2 - \cHPV  N^{3-2/p} +1$ instead of $N^2/2 \leq n\leq N^2 - b_p N^{3-2/p}\log N$ does not change anything, since the assertion follows directly from \cite[Theorem 3.4 and Proposition~3.6]{GKS1987} and it is only crucial that $n$ is of order $N^2$.
	
	\emph{Step 2.} 
	The second part follows from factorization. By duality, we may and do assume that $1/p+1/q\leq 1$, i.e., $1/q^* \geq 1/p$ and $\critmax_{p,q}=3-2/p \geq 3-2/q^* = 1+2/q$.
	Hence
		\[
	a_n\big(\Sc_p^N \hookrightarrow \Sc_q^N\big) 
	\leq \|\Sc_{p}^N \hookrightarrow \Sc_{q^*}^N \| a_n\bigl(\Sc_{q^*}^N \hookrightarrow \Sc_{q}^N  \bigr) \\
	=  N^{1/q^*  - 1/p} a_n\big(\Sc_{q^*}^N \hookrightarrow \Sc_q^N\big).
	\]
	If $(1-\cHPV )N^2 \leq n\leq N^2 - \cHPV  N^{\critmax_{p,q}} +1$, then also $(1-\cHPV )N^2 \leq n\leq N^2 - \cHPV  N^{3-2/q^*} +1$, and the assertion follows from the result of Step 1. 
\end{proof}

\begin{proof}[Proof of Proposition~\ref{prop:approximation-square-upper-bd-special}]
By duality it suffices to prove the first estimate.
By  \cite[Proposition~3.8]{GKS1987},
\begin{align*}
a_n\big(\Sc_1^N \hookrightarrow \Sc_\infty^N\big) 
&\lesssim_q
\begin{cases}
\min\Bigl\{ 1, \frac{N^{1/2}}{n^{1/2}} \Bigr\}
          &:\  1 \leq n \leq (1-\cHPV) N^2,\\
N^{-3/2} \sqrt{N^2-n+1}
          &:\ (1-\cHPV) N^2 \leq n\leq N^2 -  \cHPV N+1,\\
N^{-1} 
          &:\ N^2 - \cHPV N +1\leq n \leq N^2
\end{cases}
\end{align*}
(we may and do choose  slightly different ranges than in \cite{GKS1987} just for the sake of consistency with other results).
Thus, the assertion follows from the factorization
\[
a_n\big(\Sc_1^N \hookrightarrow \Sc_q^N\big) 
 \leq a_n\big(\Sc_{1}^N \hookrightarrow \Sc_\infty^N\big) \| \Sc_\infty^N \hookrightarrow \Sc_{q}^N  \|  \\
=  a_n\big(\Sc_{1}^N \hookrightarrow \Sc_\infty^N\big) N^{1/q}
\]
and the trivial upper bound $a_n\big(\Sc_1^N \hookrightarrow \Sc_q^N\big)\leq \| \Sc_1^N \hookrightarrow \Sc_{q}^N  \|=1 $.
\end{proof}

\section{Approximation numbers in the case $1\leq p\leq q\leq 2$ or $2\leq p\leq q  \leq \infty$} \label{sec:approximation-two-triangles}

The cases $1\leq p\leq q\leq 2$ and $2\leq p\leq q  \leq \infty$ are dual to each other:
if $2\leq p\leq q\leq \infty$, then $1\leq q^* \leq p^* \leq 2$ and $a_n\big(\Sc_p^N\hookrightarrow \Sc_q^N\big)  = a_n\big(\Sc_{q^*}^N\hookrightarrow \Sc_{p^*}^N\big)$.

Below we consider $2\leq p\leq q  \leq \infty$. The proofs of the estimates for the approximation numbers follow the reasoning for Gelfand numbers.

\begin{proposition}
\label{prop:approximation-triangle-1-upper-bd}
Let $2\leq p\leq q \leq \infty$ and assume that $n,N\in \N$ with $1\leq n \leq N^2$.
Then
\begin{align*}
a_n\big(\Sc_p^N \hookrightarrow \Sc_q^N\big)
&\lesssim_{q}
\begin{cases}
1       
           &:\  1 \leq n \leq N^2 - c(q) N^{1+2/p} +1,\\
N^{-1/2-1/p} \sqrt{N^2-n+1}
          &:\ N^2 - c(q) N^{1+2/p} +1\leq n\leq N^2 -  \cHPV N^{1+2/q}+1,\\
N^{1/q-1/p} 
          &:\ N^2 - \cHPV N^{1+2/q} +1\leq n \leq N^2.
\end{cases}
\end{align*}
Here $c(q)\in(0,\infty)$ is a constant that depends only on $q$ and $\cHPV\in(0,1)$ is the constant from \cite[Lemma 2.5]{HPV2020}.
\end{proposition}

We remark that if $2\leq p\leq q\leq\infty$, then we have the following lower bounds due to Hinrichs--Michels \cite[Example 4.14]{HM2005},
\[
a_n(\Sc_p^N \hookrightarrow \Sc_q^N)
 \geq c_n(\Sc_p^N \hookrightarrow \Sc_q^N)
\gtrsim_{p,q}
\begin{cases}
\sqrt{\frac{N^2-n+1}{N^2}}^{\frac{1/p-1/q}{1/2-1/q}}
          &:\ 1 \leq n\leq N^2 -  N^{1+2/q}+1,\\
N^{1/q-1/p} 
          &:\  N^2 - N^{1+2/q} +1\leq n \leq N^2.
\end{cases}
\]
They match the above upper bounds for $1\leq n\leq c(p,q) N^2$ and $N^2-\cHPV N^{1+2/q} + 1\leq n\leq N$.

\begin{proof}[Proof of Proposition~\ref{prop:approximation-triangle-1-upper-bd}]
As in \cite{HPV2020} in the case of Gelfand numbers, by factorization and the result of Carl--Defant~\cite{CD1997} we can write
\begin{align*} 
	a_n(\Sc_p^N \hookrightarrow \Sc_q^N) 
	& \leq \|\Sc_p^N \hookrightarrow \Sc_2^N\| a_n(\Sc_2^N \hookrightarrow \Sc_q^N) \\
	& \asymp_{q} N^{1/2 - 1/p } \max \Bigl\{ N^{1/q - 1/2}, N^{-1}\sqrt{N^2-n+1} \Bigr\}\\
	& =
	\begin{cases}
		N^{-1/2 - 1/p} (N^2-n+1)^{1/2}  &:\ 1\leq n\leq N^2 - N^{1+2/q} + 1,\\
		N^{1/q - 1/p} &:\ N^2 - N^{1+2/q} + 1 \leq n\leq N^2
	\end{cases}
\end{align*}
(and of course, by paying the price of a constant, we can take $N^2 - \cHPV N^{1+2/q} + 1$ instead of  $N^2 - N^{1+2/q} + 1$ in the definition of the ranges, just  for consistency with other statements).
We also have the trivial bound $a_n(\Sc_p^N \hookrightarrow \Sc_q^N) \leq \|\Sc_p^N \hookrightarrow \Sc_q^N\| = 1$, which is better than the above estimate whenever $n\leq N^2 - c(q) N^{1+2/p} + 1$.
\end{proof}

\section{Approximation and Kolmogorov numbers in the case $0<p \leq 1$, $p\leq q$}
\label{sec:approximation-p-quasi}

\begin{proposition}
	Let $0 < p \leq 1$, $2\leq q\leq \infty$ and assume that $n,N\in \N$ with $1\leq n \leq N^2$.
	Then
	\begin{equation*}
		a_n \big(\Sc_p^N \hookrightarrow \Sc_q^N\big) = d_n\big(\Sc_p^N \hookrightarrow \Sc_q^N\big) 
		\asymp_{q}
\begin{cases}
\min\Bigl\{ 1, \frac{N^{1/2+1/q}}{n^{1/2}} \Bigr\}
          &:\  1 \leq n \leq (1-\cHPV) N^2,\\
N^{-3/2+1/q} \sqrt{N^2-n+1}
          & :\ (1-\cHPV) N^2 \leq n\leq N^2 -  \cHPV N+1,\\
N^{1/q-1} 
          &:\ N^2 - \cHPV N +1\leq n \leq N^2.
\end{cases}
\end{equation*}
\end{proposition}

\begin{proof}
By Lemma~\ref{lem:approximation-Kolmogorov-p-quasi} the $s$-numbers in question are equal to those of the embedding $\Sc_1^N \hookrightarrow \Sc_q^N$ (in particular the constant in the estimate does not depend on $p$).
Hence, the assertion follows  from  the results of Sections~\ref{sec:Gelfand} and \ref{sec:approximation-square} and  \cite{HPV2020}.
\end{proof}

\begin{proposition}
	Let $0 < p \leq 1\leq q\leq 2$ and assume that $n,N\in \N$ with $1\leq n \leq N^2$.
	Then
	\begin{equation*}
		d_n\big(\Sc_p^N \hookrightarrow \Sc_q^N\big)\leq 
		a_n (\Sc_p^N \hookrightarrow \Sc_q^N) \lesssim_q
		\begin{cases}
			1       
			&:\,  1 \leq n \leq N^2 - c(1) N^{3-2/q} +1,\\
			N^{-3/2 + 1/q} \sqrt{N^2-n+1}
			&:\, N^2 - c(1) N^{3-2/q} +1\leq n\leq N^2 -  \cHPV N+1,\\
			N^{1/q-1} 
			&:\, N^2 - \cHPV N +1\leq n \leq N^2,
		\end{cases}
	\end{equation*}
	and
	\begin{equation*}
		a_n\big(\Sc_p^N \hookrightarrow \Sc_q^N\big) \geq
		d_n (\Sc_p^N \hookrightarrow \Sc_q^N)
		\gtrsim_q
		\begin{cases}
			1       
			&:\,  1 \leq n \leq (1-\cHPV)N^2,\\
			\Bigl(\frac{N^2-n+1}{N^2}\Bigr)^{1-1/q}
			&:\, (1-\cHPV)N^2 \leq n\leq N^2 - \cHPV N+1,\\
			N^{1/q-1} 
			&:\, N^2 - \cHPV N +1\leq n \leq N^2.
		\end{cases}
	\end{equation*}
\end{proposition}

\begin{proof}
	Again by Lemma~\ref{lem:approximation-Kolmogorov-p-quasi} the $s$-numbers in question are equal to those of the embedding $\Sc_1^N \hookrightarrow \Sc_q^N$ (in particular the constant in the estimate does not depend on $p$), but this time in the results of  Section~\ref{sec:approximation-two-triangles} and \cite{HPV2020} we have some gaps in the intermediate range.
\end{proof}

\begin{proposition}
	\label{prop:approximation-Kolmogorov-1}
	Let $0 < p \leq q \leq 1$ and assume that $n,N\in \N$ with $1\leq n \leq N^2$.
	Then
	\begin{equation*}
		a_n (\Sc_p^N \hookrightarrow \Sc_q^N) = d_n(\Sc_p^N \hookrightarrow \Sc_q^N) = 1.
	\end{equation*}
\end{proposition}

\begin{proof}  
	For $0 < p \leq q \leq 1$ and $1\leq n \leq N^2$ we have
	\begin{align*}
	1 = d_n(\Sc_1^N \hookrightarrow \Sc_1^N) = d_n(\Sc_p^N \hookrightarrow \Sc_1^N)
	  &\leq d_n(\Sc_p^N \hookrightarrow \Sc_q^N)\\
	  &\leq a_n(\Sc_p^N \hookrightarrow \Sc_q^N) \leq a_n(\Sc_p^N \hookrightarrow \Sc_p^N) =1,
	\end{align*}
where we used factorization (note that $\| \Sc_q^N \hookrightarrow \Sc_1^N \| = \| \Sc_p^N \hookrightarrow \Sc_q^N \| = 1$ for $0<p\leq q\leq 1$) and Lemma~\ref{lem:approximation-Kolmogorov-id-1}.
	\end{proof}


\subsection*{Acknowledgement}
J. Prochno and M. Strzelecki are supported by the Austrian Science Fund (FWF) Project P32405
``Asymptotic geometric analysis and applications''. J. Prochno is also supported  by Project F5513-N26 of the Austrian Science Fund,  
which  is  a  part  of  the  Special 
Research  Program  ``Quasi-Monte  Carlo  Methods:  Theory  and  Applications''.

\bibliographystyle{plain}
\bibliography{approx_schatten}

\bigskip
	
	\bigskip
	
	\medskip
	
	\small

	\noindent \textsc{Joscha Prochno:} Institute of Mathematics and Scientific Computing,
	University of Graz, Heinrichstrasse 36, 8010 Graz, Austria
	
	\noindent 
	{\it E-mail:} \texttt{joscha.prochno@uni-graz.at}

		\medskip
		
		\noindent \textsc{Micha{\l} Strzelecki:} Institute of Mathematics and Scientific Computing,
			University of Graz, Heinrichstrasse 36, 8010 Graz, Austria
		
		\noindent 
		{\it E-mail:} \texttt{michal.strzelecki@uni-graz.at}

\end{document}